\documentclass[10pt]{amsart}
\usepackage{amsfonts}

\usepackage{amsmath}
\usepackage{graphics}
\usepackage{epsfig}
\usepackage[colorlinks=true]{hyperref}
\usepackage{float}
\usepackage{caption}
\usepackage{subfigure}
\usepackage{wrapfig}
\usepackage{booktabs}

\newtheorem{theorem}{Theorem}[section]
\newtheorem{corollary}{Corollary}

\theoremstyle{definition}

\newtheorem{remark}{Remark}

\title[blow-up in delay model ]
{Finite time Blow up in a population model with competitive interference and time delay}
\author[  Parshad, Bhowmick, Quansah, Agrawal, Upadhyay. ]{}

 \email{ranjit.chaos@gmail.com}

\begin{document}
\maketitle

\centerline{\scshape 
 Rana D. Parshad, Suman Bhowmick and Emmanuel Quansah}
{
\footnotesize
 \centerline{Department of Mathematics,}
 \centerline{Clarkson University,}
   \centerline{ Potsdam, New York 13699, USA.}
 }

   \medskip
\centerline{\scshape Rashmi Agrawal and Ranjit Kumar Upadhyay }
{\footnotesize
 \centerline{Department of Applied Mathematics,}
 \centerline{Indian School of Mines,}
   \centerline{ Dhanbad 826004, Jharkhand, India.}
}

\begin{abstract}
In the current manuscript, an attempt has been made to understand the dynamics of a time-delayed predator-prey system with modified Leslie-Gower and Beddington-DeAngelis type functional responses for large initial data. In \cite{RK15}, we have seen that the model does possess globally bounded solutions, for small initial conditions, under certain parametric restrictions. Here, we show that actually solutions to this model system can blow-up in finite time, for large initial condition, \emph{even} under the parametric restrictions derived in \cite{RK15}. We prove blow-up in the delayed model, as well as the non delayed model, providing sufficient conditions on the largeness of data, required for finite time blow-up. Numerical simulations show, that actually the initial data does not have to be very large, to induce blow-up. The spatially explicit system is seen to possess Turing instability. We have also studied Hopf-bifurcation direction in the spatial system, as well as stability of the spatial Hopf-bifurcation using the central manifold theorem and normal form theory. 
\end{abstract}

\section{Introduction}

We live in a finite world which can only sustain a finite number (density) of individuals. Any realistic ecological model should be dissipative, i.e., populations densities are eventually bounded by a constant independent of initial conditions. May (\cite{RM76},\cite{RM82}) suggested that in an orgy of reciprocal benefication, Lotka-Volterra (LV) type models with sufficiently strong mutualistic interactions, populations can blow-up in finite time despite being self-limited. Lin (\cite{Lin02}) obtained the lower and upper bounds on the blow up rate for a  mutualistic LV model. Zhou and Lin \cite{ZL12} studied the finite time blow-up and global existence of a nonlocal problem in space with a free boundary condition. Their results show that the blow up occurs for sufficiently large values of initial conditions while the global fast solution exists for sufficiently small initial data, and the intermediate case with suitably large initial data yields a global slow solution. Ninomiya and Weinberger \cite{NH03} showed that in a simple two species predator-prey like model, populations always converges to a unique equilibrium state, but if the death rate of the predator is increased ever so slightly, then there are initial conditions for which both populations blow-up in finite time. Stuart and Floater \cite{AM90} studied the blow up problem in detail for nonlinear scalar ODE and showed that fixed step method fail to compute the blow up time. Based on arc length transformation, Hirota and Ozawa \cite{CK06} presented a numerical method for estimating the blow-up time of the solution of ODE. Both (\cite{AM90},\cite{CK06}) studied blow up problems in PDE by discretizing  the equations in space, and integrating the resulting ODE system  by a method of lines approach. Fila et al. \cite{MH06} discuss the blow up results of the scalar ODE and suggested that Dirichlet boundary condition can in fact prevent blow up in reaction-diffusion systems.
Parshad et al. in a series of works \cite{RD13, PKK15} investigated a three species food chain model with generalist top predator and observed that the model exhibits finite time blow up in certain parameter range and for large initial data.  \\  
\\
    Upadhyay and Iyengar \cite{RK13} considered the case of an environment where there is a prey and a generalist predator. A two species system with a modified Leslie-Gower  and Beddington-DeAngelis type functional responses is proposed via the following ODE model that describes the interactions between a generalist predator $Y$ and its most favorite food, 
	a prey species $X$.
\begin{eqnarray}
\label{eq:1}
\frac{dX}{dt}&=&rX(1-\frac{X}{K})-\frac{\omega XY}{D+dX+Y}\\
\frac{dY}{dt}&=&cY^2-\frac{\omega_{1}Y^2}{X+D_{1}} \nonumber
\end{eqnarray}

It is assumed that the prey grows logistically in the absence of predator with per capita rate $r$, and carrying capacity $K$, $\omega$ is the maximum rate of per capita removal of prey species $X$, due to predation by a generalist predator $Y$. $d$ and $D$ quantify the extent to which environment provides protection to the prey $X$, and may be thought of as a refuge or a measure of the effectiveness of the prey, in evading the predators attack. $c$ is the growth rate of the generalist predator $Y$, due to sexual reproduction, and $Y^2$ signifies the fact that mating frequency is directly proportional to the number of males as well as female individuals. $\omega_{1}$ is the maximum rate of per capita removal of predator species $Y$, and $D_{1}$ normalizes the residual reduction in the predator population because of severe scarcity of its favorite food $X$.
It is well known that the reproduction of a predator population after depredating on the prey will not be instantaneous, but mediated by some constant time lag $\tau> 0$ for gestation of predator, i.e., $\tau$ represents the time lag required for gestation of predator which is based on the assumption that the rate of change in predator population depends on the number of prey and of predator at some previous time. Based on the above, recently Upadhyay and Agrawal \cite{RK15}  propose that the dynamics of the system \eqref{eq:1} are governed by the following system of nonlinear delay differential equations:

\begin{eqnarray}
\label{eq:1d}
\frac{dX}{dt}&=&rX(1-\frac{X}{K})-\frac{\omega XY}{D+dX+Y} \nonumber \\
\frac{dY}{dt}&=& Y\left(cY-\frac{\omega_{1}Y(t-\tau)}{X(t-\tau)+D_{1}}\right)
\end{eqnarray}

The initial conditions associated with \eqref{eq:1d} are 

\begin{eqnarray}
X(\theta)=\phi_{1}(\theta) \geq 0, Y(\theta)= \phi_{2}(\theta) \geq 0, \theta \in [-\tau, 0],
\end{eqnarray}
such that $\phi_{i}(0) >0 (i=1,2)$,
where $\phi : [-\tau,0] \longrightarrow \mathbb{R}^{2}$ with norm 

\begin{eqnarray}
||\phi|| = sup_{-\tau \leq \theta \leq 0} (|\phi_{1}(\theta)|,|\phi_{2}(\theta)|) 
\end{eqnarray}
such that $\phi=(\phi_{1},\phi_{2})$.

In this work, the authors \cite{RK15} prove results on positive invariance, boundedness, existence of equilibria, direction and stability of Hopf-bifurcation. They also derived conditions for global stability using a suitable Lyapunov function and suggest how finite time blow-up can be prevented in such a system. Numerically, they have shown the global stability  of the interior equilibrium  point (c.f. Fig 3) for a given set of biologically realist parameter values. This global stability result is valid for a small range of initial conditions up to $(13.5, 13.5)$. After that a slight perturbation in the initial condition for example to $(13.6, 13.6)$, invalidates the global stability result. In the current manuscript, we explore the other side of the story. That is the solutions of the model system \eqref{eq:1d} can blow-up in finite time, if the initial data is sufficiently large, even under the parameter restrictions imposed in \cite{RK15} for global stability result.  Here we will prove the finite time blow up results for the delay and non-delay model system both analytically and numerically. We also study the spatial counterpart of this delay model \eqref{eq:1d}. For the delayed diffusive model \eqref{eq:3.1}-\eqref{eq:3.1a}, we will study the local stability, existence of Hopf-bifurcation and stability and direction of periodic solutions using the normal form and center manifold theorem. Recently, Xu and Yuan \cite{XY15} studied the spatial periodic solutions in a delay diffusive predator-prey model with herd behaviour. Li and Wang \cite{LW15} found that the critical value of time delay affects the stability of the positive constant equilibrium and the condition of occurrence of Hopf-bifurcation is stronger due to the emergence of diffusion in a delayed diffusive models.

\section{Finite time blow-up in the delayed model system}

\label{thm:u1}




In the current manuscript, we show the following
\begin{itemize}
\item \emph{There is no global stability} for model system \eqref{eq:1d}. In fact solutions to system \eqref{eq:1d} can blow-up in finite time if the initial data is sufficiently large, even under the parametric restrictions imposed in Theorem 5  \cite{RK15}. This is shown via Theorem \ref{thm:t1}.

\item Solutions to the non delayed model \eqref{eq:1} can also blow up in finite time, if the initial data is sufficiently large. This is shown via Theorem \ref{thm:t2}.

\item Numerical simulations confirm the above. In particular we show that in reality, the initial data does not have to be to large to induce blow-up. This is shown in Figure  \ref{fig:1t}. 

\item Interestingly, the delayed model \eqref{eq:1d} can be shown to blow-up in finite time, even if the $cY^2$ term in the predator equation in \eqref{eq:1d} is replaced by $cY^m$, where $1 < m < 2$. This is \emph{not} possible in the non delayed model \eqref{eq:1}. This makes these two models completely different in this respect. This is shown via Corollary \ref{cor:t11}, and a numerical simulation is shown to confirm this via Figure \ref{fig:1t1}.

\item The spatially explicit model system \eqref{eq:1di}-\eqref{eq:1di2} possesses Turing instability. This is shown in Theorem \ref{thm:tur} and via Figs. \ref{fig:1tt} - \ref{fig:2t}. 

\item The stability and direction of Hopf bifurcation in the spatial delayed system \eqref{eq:1d} is presented via Theorem \ref{thm:hop}.
\end{itemize}

We first state and prove the finite time blow-up result for the time delay model,

\begin{theorem}
\label{thm:t1}
Consider the delayed system \eqref{eq:1d} in which the interior equilibrium $E_{2}$ of the system \eqref{eq:1d} is locally asymptotically stable. For sufficiently large initial data,  interior equilibrium $E_{2}$ of the system \eqref{eq:1d} is not globally asymptotically stable, even if the parametric restrictions in Theorem 5 \cite{RK15} are met. In particular, under the restrictions of Theorem 5 \cite{RK15}, solutions to \eqref{eq:1d} can blow up in finite time, for sufficiently large initial data. That is

\begin{equation*}
\lim_{t\rightarrow T^{\ast}<\infty}\| Y \| \rightarrow \infty.
 \end{equation*}

\end{theorem}

\begin{proof}
We begin with a lower estimate for $X$. Clearly

\begin{eqnarray}
\label{eq:2}
&& \frac{dX}{dt} = rX(1-\frac{X}{K})-\frac{\omega XY}{D+dX+Y}  \nonumber \\
&& \geq rX - rX(\frac{K}{K}) - \omega X \nonumber \\
&& = - \omega X. 
\end{eqnarray}


This follows because via a simple comparison argument, by comparing the solution of the prey equation in \eqref{eq:1d}, to the ODE $ \frac{dX}{dt} = rX(1-\frac{X}{K})$, we obtain that the value of $X$ solving \eqref{eq:1d} satisfies $X \leq K$. Furthermore $\frac{\omega Y}{D+dX+Y} \leq \omega$. Thus integrating \eqref{eq:2}in the time interval $[0,t-\tau]$, we obtain

\begin{eqnarray}
X(t-\tau) \geq X(0)e^{-\omega(t-\tau)}. 
\end{eqnarray}
Thus

\begin{eqnarray}
\frac{-1}{X(t-\tau)} \geq -\frac{e^{\omega(t-\tau)}}{X(0)}. 
\end{eqnarray}

For our purposes we assume constant history or IC, that is 

\begin{eqnarray}
X(\theta)&=&\phi_{1}(\theta) = X_{0} \geq 0, Y(\theta)= \phi_{2}(\theta) =Y_{0} \geq 0,\nonumber\\ 
\theta &\in& [-\tau, 0], \phi_{1}(0) =X_{0} >0, \phi_{2}(0) =Y_{0} >0.\nonumber\\
\end{eqnarray}

Now we consider the equation for the predator $Y$ in \eqref{eq:1d}. Note

\begin{eqnarray}
\label{eq:y1}
&&\frac{dY}{dt} = Y\left(cY-\frac{\omega_{1}Y(t-\tau)}{X(t-\tau)+D_{1}}\right) \nonumber \\
&& \geq cY^2 - \left(\frac{\omega_{1}Y Y(t-\tau)}{X(t-\tau)}\right) \nonumber \\
&& \geq cY^2 - \omega_{1}Y Y(t-\tau)\frac{e^{\omega(t-\tau)}}{X(0)}. \nonumber \\
\end{eqnarray}

Let us now assume that $Y$ is bounded on $[-\tau, t_{1}]$, where $t_{1}$ can be arbitrarily large. That is $Y<M, t \in [-\tau, t_{1}]$.Then we have 

\begin{eqnarray}
 |Y(t-\tau)| < \max(M,\phi_{2}(\theta)=Y_{0}), t \in [-\tau, t_{1}].
\end{eqnarray}

Thus

\begin{eqnarray}
&&\frac{dY}{dt} \geq cY^2 - \omega_{1}Y Y(t-\tau)\frac{e^{\omega(t-\tau)}}{X(0)}  \nonumber \\
&& >  cY^2  - \omega_{1}M\frac{e^{\omega(t_{1}-\tau)}Y}{X(0)}  \nonumber \\
&& =  cY^2  - \omega_{1}M\frac{CY}{X(0)}.  \nonumber \\
\end{eqnarray}
Here $C$ is an upper bound on $e^{\omega(t_{1}-\tau)}$. Thus we obtain,

\begin{equation}
\label{eq:1c}
\frac{dY}{dt} \geq cY^2  - \omega_{1}M\frac{C Y}{X(0)}. 
\end{equation}

However, the solution to equation \eqref{eq:1c} clearly blows up in finite time, if the initial data $Y_{0}$ is large enough, for a fixed $X_{0}$. 

To see this we can compare to the equation 

\begin{equation}
\label{eq:1b}
\frac{dY}{dt} =   cY^2  - \omega_{1}M\frac{C Y}{X(0)},
\end{equation}

which blows up as long as the initial data $Y_0$ is larger than $\omega_{1}M\frac{C}{c |X(0)|}$, and $c>0$.
 Furthermore, by choosing $X_{0} >>1$, we can induce blow-up for $Y_{0}$ small as well.
Thus the $Y$ solving \eqref{eq:1c} and \eqref{eq:1d} by simple comparison also blow-up in finite time.
This tells us that $Y$ cannot be bounded on any arbitrary time interval $[-\tau,t_{1}]$, and must blow up at some finite time $T^{*} <t_{1} <\infty$.
 This proves the Theorem.
\end{proof}

\begin{remark}
Note that in the above theorem the only condition required for blow-up, is the positivity of $c$, and the largeness of the initial data. These \emph{do not} hinder the parametric restrictions required in Theorem 5 \cite{RK15}. Thus $Y$ will blow-up for initial data chosen large enough, even under the parametric restrictions of Theorem 5 \cite{RK15}.
\end{remark}

We next state and prove a finite time blow-up result for the time delayed model, with slight modifications.

\begin{corollary}
\label{cor:t11}
Consider the delayed system \eqref{eq:1d}, with the $cY^2$ term replaced by $cY^m$, $1<m<2$. Even if the parametric restrictions in Theorem 5 are met solutions to the modified model can blow up in finite time, for sufficiently large initial data. That is

\begin{equation*}
\lim_{t\rightarrow T^{\ast}<\infty}\| Y \| \rightarrow \infty.
 \end{equation*}

\end{corollary}

\begin{proof}
The proof follows trivially by observing that just as in \eqref{eq:1b}, the solution to

\begin{equation}
\label{eq:1bnn}
\frac{dY}{dt} =   cY^m  - \omega_{1}M\frac{C Y}{X(0)}, \ 1<m<2, 
\end{equation}

will also blow-up in finite time for large enough initial data. This proves the corollary.

\end{proof}

We next state and prove the finite time blow-up result for the non delayed model.

\begin{theorem}
\label{thm:t2}
Consider the two species model given by equation \eqref{eq:1}, for any choice of parameters, including the ones satisfying Theorem 5, and a $\delta_{1} > 0$, such that $c > \delta_{1}$.  Given any initial data $Y_{0}$, there exists initial data $X_{0}$, such that if this data meets the largeness condition

\begin{equation}
\frac{\omega}{\delta_{1}} < |Y_{0}| \ln\left( \frac{|X_{0}|}{\frac{\omega}{c - \delta_{1}} - D_{1}}\right) 
\end{equation}

then \eqref{eq:1} will blow-up in finite time, that is

\begin{equation*}
\lim_{t\rightarrow T^{\ast}<\infty}\| Y \| \rightarrow \infty.
 \end{equation*}

Here the blow-up time $T^{*} \leq \frac{1}{\delta_{1}|Y_{0}|}$

\end{theorem}

\begin{proof}
Consider the equation for the predator

\begin{equation*}
\frac{dY}{dt}=\left(c-\frac{\omega}{X+\omega_{12}}\right)Y^{2}.
\end{equation*}

In the event that $c> \frac{\omega}{D_{1}}$, blow-up is trivial. If $c < \frac{\omega}{D_{1}}$,  blow-up is far from obvious. However still possible for large data. To see this note, if

\begin{equation*}
\left(c-\frac{\omega}{X+D_{1}}\right) > \delta_{1} > 0,
\end{equation*}

then $Y$ will blow-up in finite time in comparison with

\begin{equation*}
\frac{dY}{dt}=\delta_{1} Y^{2}
\end{equation*}

The tricky part here is that $\left(c-\frac{\omega}{X+D_{1}}\right)$ can switch sign, and this is dependent on the dynamics of the prey $X$, which changes in time.
In order to guarantee blow-up, we must have that $\left(c-\frac{\omega}{X+D_{1}}\right) > \delta_{1} $ or equivalently we must guarantee that

\begin{equation}
\label{eq:ve}
X > \frac{\omega}{c- \delta_{1}} - D_{1}. 
\end{equation}

To this end we will work with the equation for the prey $X$,

\begin{equation*}
\frac{dX}{dt} = rX(1-\frac{X}{K})-\frac{\omega XY}{D+dX+Y}
\end{equation*}

as earlier, we can make the following lower estimate on $X$

\begin{equation*}
X(t) \geq X(0)e^{-\omega(t)}
\end{equation*}

we now use this in \eqref{eq:ve} to yield

\begin{equation*}
|X| >  |X_{0}|e^{-\omega(t)}> \frac{\omega}{c- \delta_{1}} - D_{1}.
\end{equation*}

Equivalently, 

\begin{equation*}
 \ln\left( \frac{|X_{0}|}{\frac{\omega}{c- \delta_{1}} - D_{1}} \right) > t \omega.
\end{equation*}

 Now note that

\begin{equation*}
\frac{dY}{dt}=\delta_{1} Y^{2}
\end{equation*}

blows-up at time $T^{*} = \frac{1}{\delta_{1}|Y_{0}|}$, thus if we choose data such that

\begin{equation*}
  \ln\left( \frac{|X_{0}|}{\frac{\omega}{c- \delta_{1}} - D_{1}} \right) \frac{1}{\omega}> t > T^{*} = \frac{1}{\delta_{1}|Y_{0}|},
\end{equation*}

Then the above guarantees that $X$ will remain above the critical level $ \frac{\omega}{c- \delta_{1}} - D_{1}$, for sufficiently long enough time, for $y$ to blow-up. This yields that as long as the following holds

\begin{equation*}
|Y_{0}|   \ln\left( \frac{|X_{0}|}{\frac{\omega}{c- \delta_{1}} - D_{1}} \right) >\frac{\omega}{\delta_{1}},
\end{equation*}

$Y$ will blow up in finite time, independent of the choice of the parameters of Theorem 5. This proves the theorem.

\end{proof}

\section{Numerical Simulations to Demonstrate Finite Time Blow-Up}
In this section we numerically simulate \eqref{eq:1d}, for the same parameters chosen in \cite{RK15}. That is, $r=1, K=100, \omega=1, D=1.01, d=0.01, c=0.01, \omega_{1}=0.2, D_{1}=10$. Note these parameters satisfy the conditions of Theorem 5. For these parameters $E_{2}=(10,9.9)$. The simulations in \cite{RK15} are done assuming initial data $(5,5)$. We note that the dynamics claimed in \cite{RK15} are preserved as long as initial data is increased unto $(13,13)$. Once we consider initial data 
$(14,14)$ or larger, we see that finite time blow-up occurs at time slightly past $t=20.47$. This is illustrated in Figure \ref{fig:1t}.

\begin{figure}[htb]
{
\includegraphics[scale=0.15]{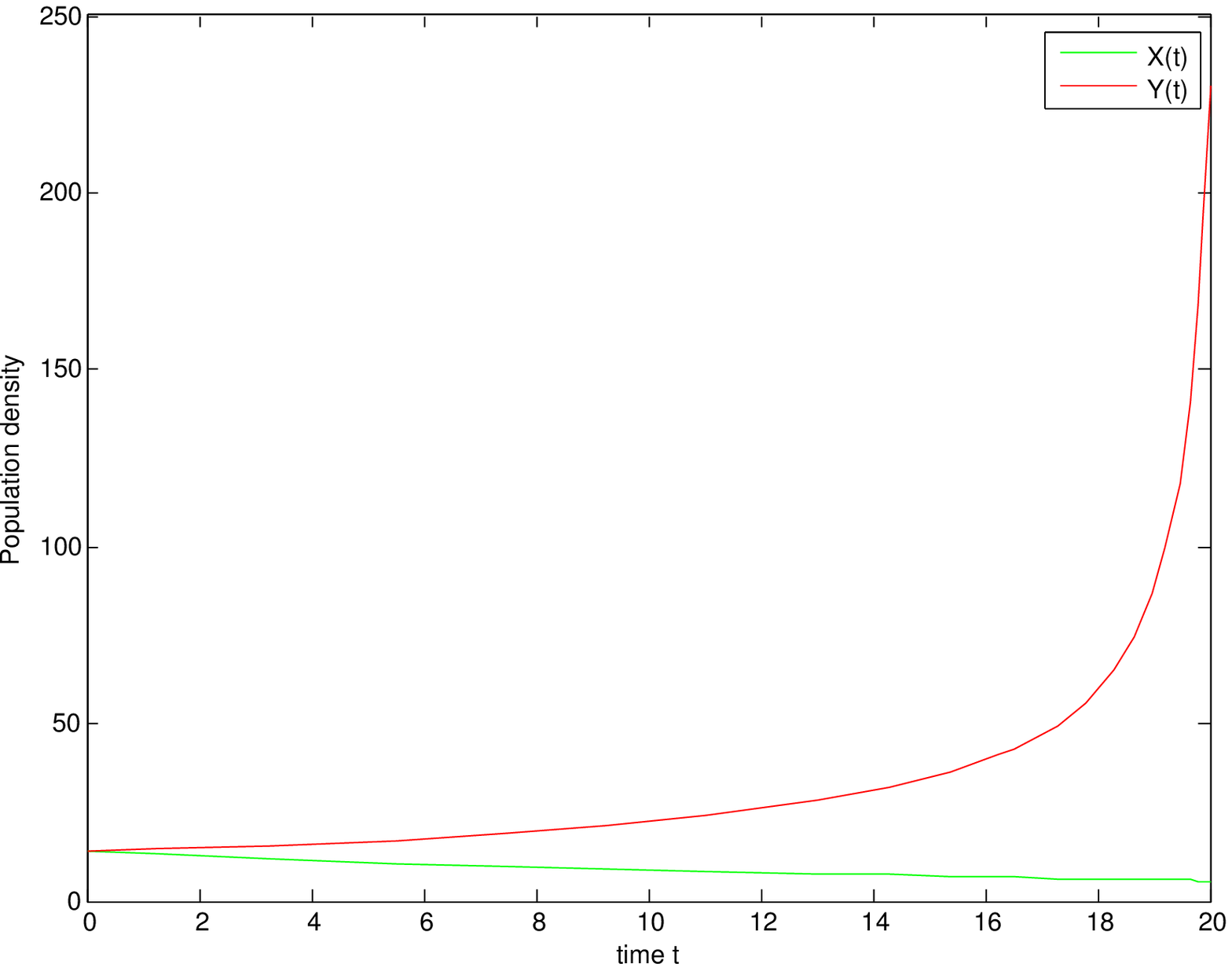}}
{
\includegraphics[scale=0.15]{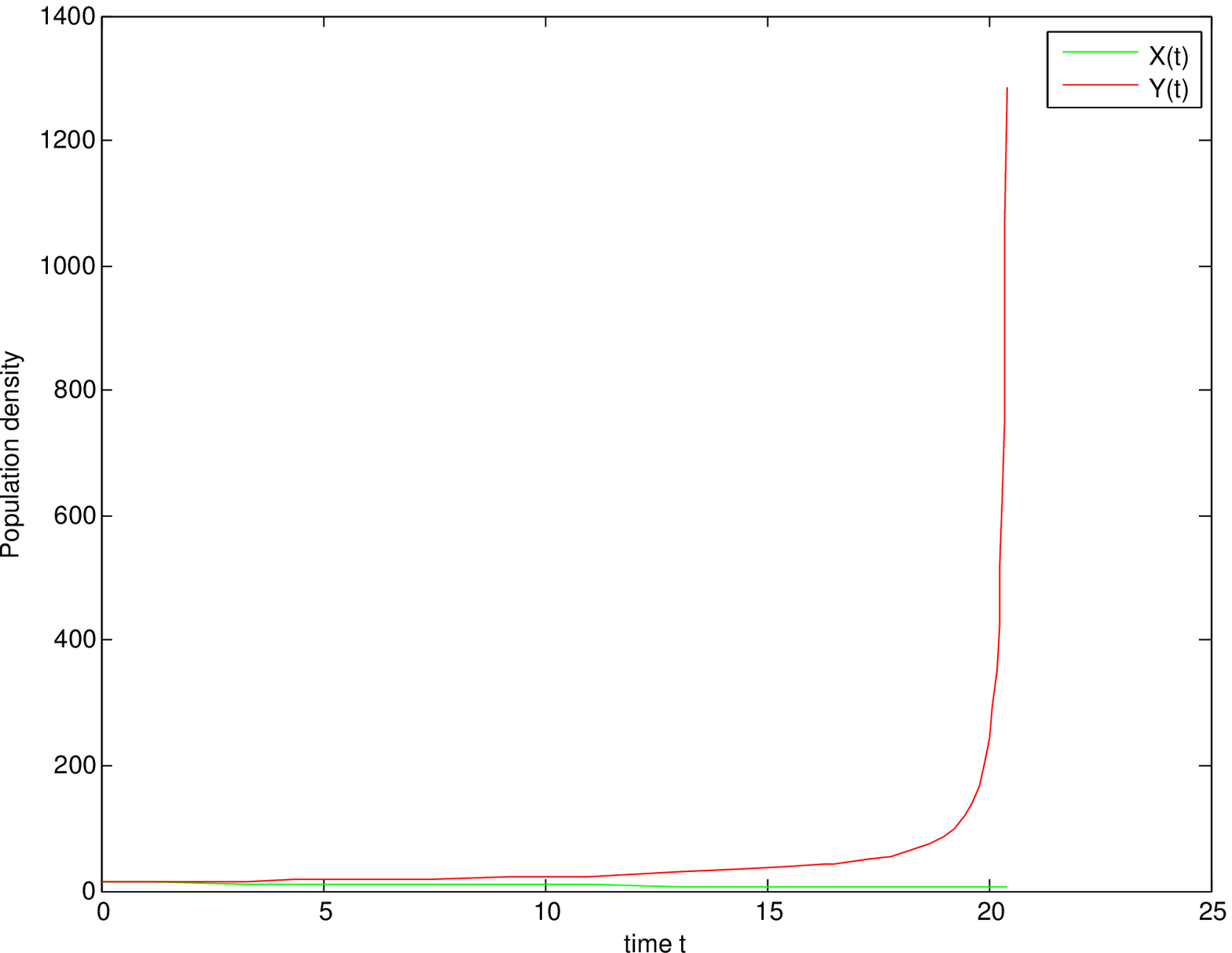}}
{
\includegraphics[scale=0.15]{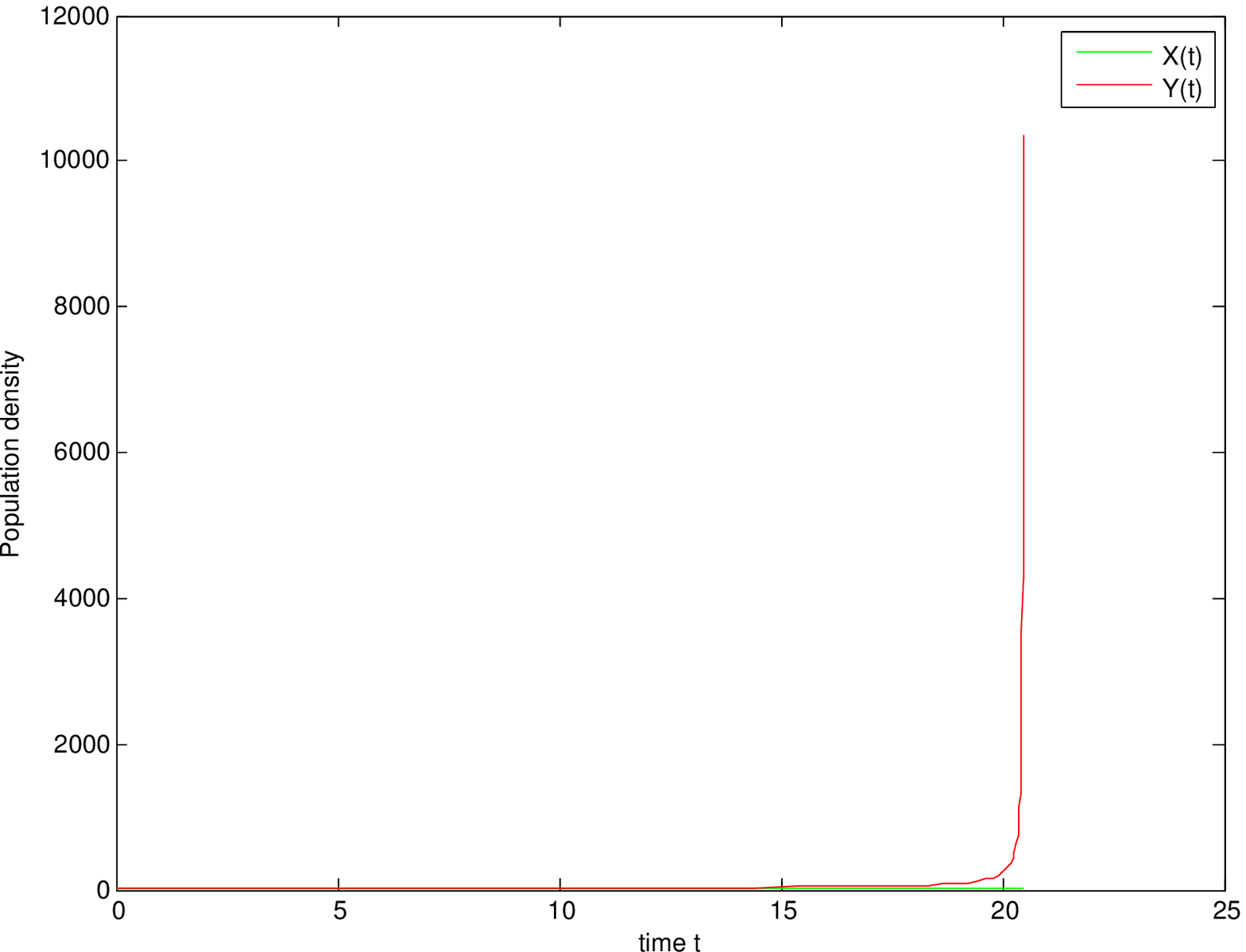}}
\caption{The populations for the prey and predator species are shown for an increasing sequence of times $t=20$, $t=20.4$ and $t=20.47$. We see that the predator population $Y(t)$ is blowing-up. Here we choose initial data $(14,14)$.}
\label{fig:1t}
\end{figure}

Next we simulate 
 \eqref{eq:1d}, for the same parameters chosen in \cite{RK15}. That is, $r=1, K=100, \omega=1, D=1.01, d=0.01, c=0.01, \omega_{1}=0.2, D_{1}=10$. We replace the $cY^2$ term in the predator equation, by $cY^{1.6}$. We still see finite time blow-up occurs for large initial data. Note, this is \emph{not} possible in the no delay case $(\tau=0)$.

\begin{figure}[htb]
{
\includegraphics[scale=0.15]{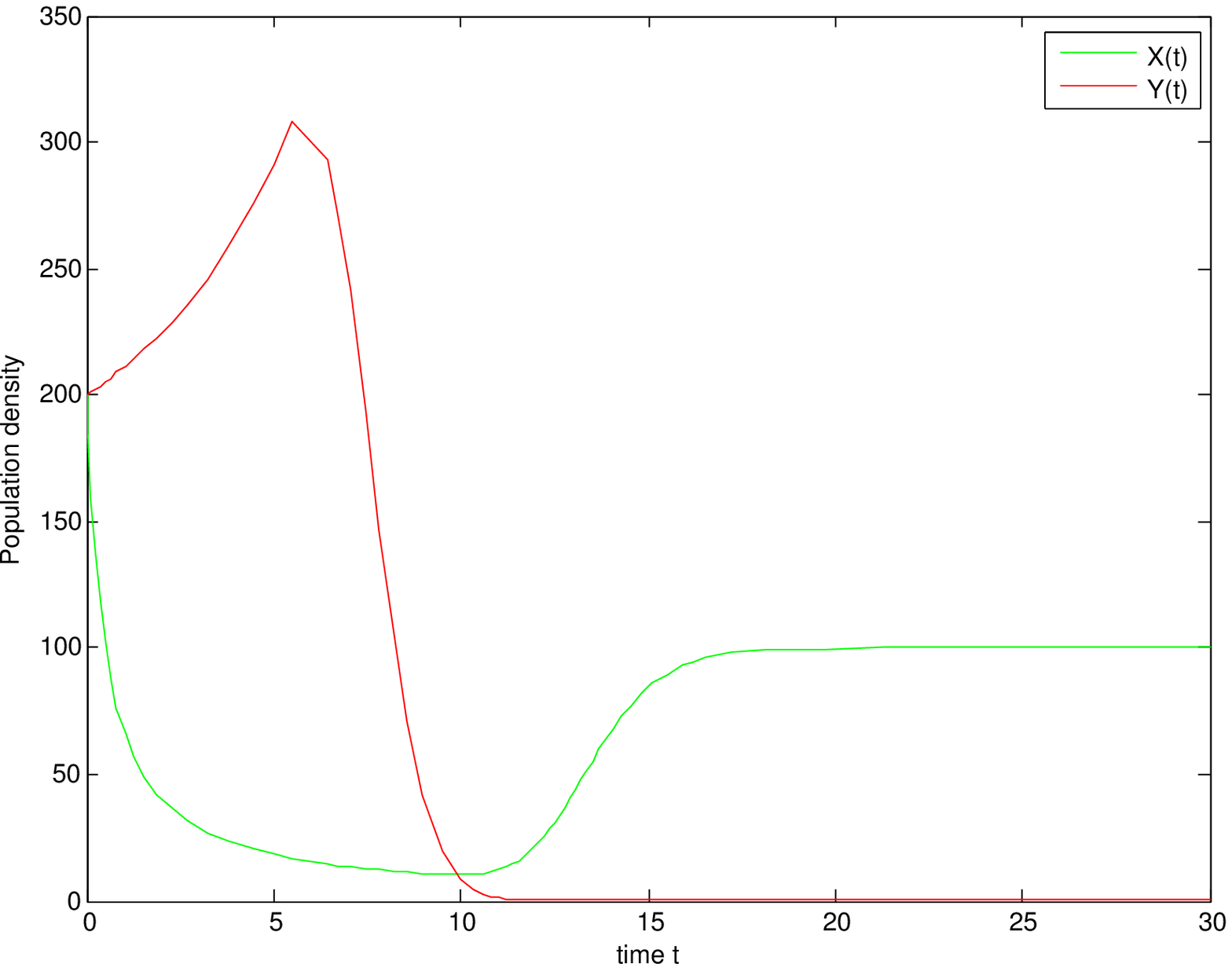}}
{
\includegraphics[scale=0.15]{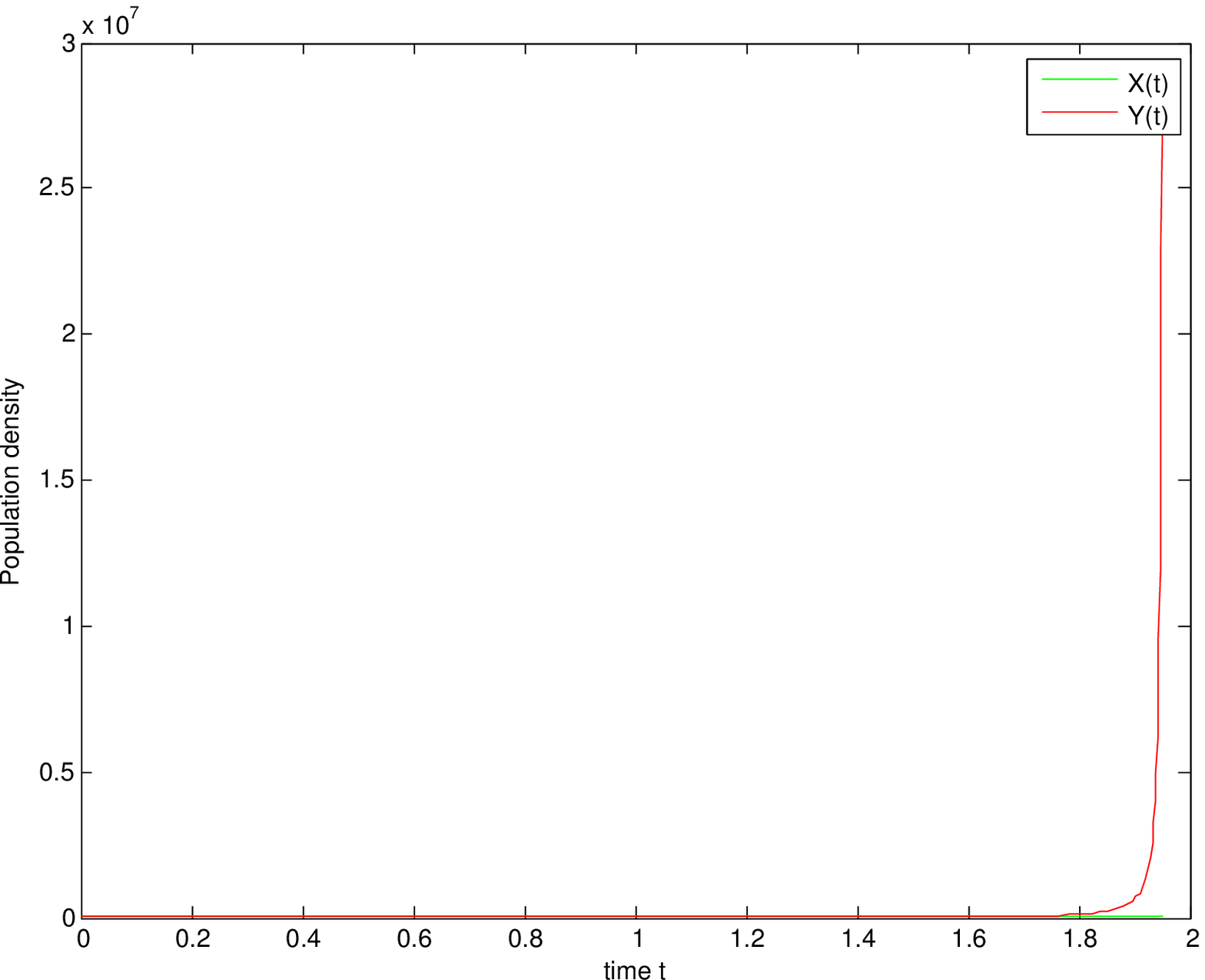}}

\caption{The left panel shows model  \eqref{eq:1d} $cY^2$ term in the predator equation, by $cY^{1.6}$. We choose initial data to be $(200,200)$. What we observe here is that the predator population starts to increase initially, but is pulled back down. The right panel shows a simulation of the same model, but when we choose initial data to be $(2000,2000)$. The populations for the prey and predator species are shown at $t=1.95$. Now we clearly see that the predator population $Y(t)$ is blowing-up. In fact it blows-up at about $t=1.955$. }
\label{fig:1t1}
\end{figure}

\section{Turing Instability}
To formulate the spatially explicit form of the earlier described ODE models, we assume that the predator and prey populations move actively in space. Random movement of animals occurs because of various requirements and necessities like, search for better food, better opportunity for social interactions such as finding mates \cite{S97}. Food availability and living conditions demand that these animals migrate to other spatial locations. In the proposed model, we have included diffusion assuming that the animal movements are uniformly distributed in all directions. 
The model with diffusion are described as follows

\begin{align}
\label{eq:1di}
\frac{\partial X}{\partial t}=d_1 \Delta X + X\Bigg[r\left(1-\frac{X}{K}\right)-\frac{\omega  Y}{D+dX+Y}\Bigg],
\end{align}
\begin{align}
\label{eq:1di2}
\frac{\partial Y}{\partial t}=d_2 \Delta Y + Y\Bigg[cY-\frac{\omega_1 Y}{X+D_1}\Bigg].
\end{align}

In this section we investigate Turing instability in \eqref{eq:1di}-\eqref{eq:1di2}. We uncover both spatial and spatio-temporal patterns, and provide the details of the Turing analysis. 
We derive conditions where the unique positive interior equilibrium point $(X^*,Y^*)$ is stable in the absence of diffusion, and unstable due to the action of diffusion, with a small perturbation to the positive interior equilibrium point. We first linearize model \eqref{eq:1di}-\eqref{eq:1di2} about the homogeneous steady state, we introduce both space and time-dependent fluctuations around $(X^*,Y^*)$. This is given as
\begin{subequations}\label{eq:7}
\begin{align}
X=X^* +  \hat{X}(\xi,t),\\
Y=Y^* + \hat{Y}(\xi,t),\\
\end{align}
\end{subequations}
where $| \hat{X}(\xi,t)|\ll X^*$, $| \hat{Y}(\xi,t)|\ll Y^*$. Conventionally we choose
\[
\left[ {\begin{array}{cc}
\hat{X}(\xi,t)  \\
\hat{Y}(\xi,t) \\
\end{array} } \right]
=
\left[ {\begin{array}{cc}
\epsilon_1  \\
\epsilon_2 \\
\end{array} } \right]
e^{\lambda t + ik\xi},
\]
where  $\epsilon_i$ for $i=1,2$ are the corresponding amplitudes, $k$ is the wavenumber, $\lambda$ is the growth rate of perturbation in time $t$ and $\xi$ is the spatial coordinate.
Substituting \eqref{eq:7} into \eqref{eq:1di}-\eqref{eq:1di2} and ignoring higher order terms including nonlinear terms, we obtain the characteristic equation  which  is given as
\begin{align}\label{eq:1.2.10}
({\bf J} - \lambda{\bf I} - k^2{\bf D})
\left[ {\begin{array}{cc}
\epsilon_1  \\
\epsilon_2 \\
\end{array} } \right]=0,
\end{align}
where
\[
\quad
\bf {D} =
\left[ {\begin{array}{cc}
d_1 & 0      \\
0     & d_2  \\
\end{array} } \right],
\]

\begin{equation}
\label{eq:J}
 \bf{J}= \begin{bmatrix}
     X^{*}\left(-\frac{r}{K}-\frac{\omega d y^{*}}{(D+dX^{*}+Y^{*})^2}\right)&   \frac{\omega X^*(D+dX^*)}{(D+dX^{*}+Y^{*})^2}  \\
         \frac{\omega_1Y^{*^{2}}}{(X^*+D1)^2}& cY^*-\frac{\omega_1Y^*}{X^*+D_1}    \\
            \end{bmatrix}
      =\begin{bmatrix}
       J_{11} & J_{12} \\
       J_{21} & J_{22} \\
        \end{bmatrix},
        \end{equation}
    \\
\\and $\bf{I}$ is a $2\times 2$ identity matrix.\\
For the non-trivial solution of \eqref{eq:1.2.10}, we require that
\[
\left|
\begin{array}{cc}
J_{11}-\lambda -k^2d_1 & J_{12}                         \\
       J_{21}                     & J_{22}-\lambda -k^2d_2 \\
      \\
 \end{array} \right|=0,
\]
This gives us a polynomial in $k$, and via standard theory \cite{M93}, we derive the following necessary and sufficient conditions for Turing Instability;
\begin{eqnarray}
\label{eq:cond}
J_{11}+J_{22}<0\\
J_{11}J_{22}-J_{12}J_{21}>0\\
d_{1}J_{22}+d_{2}J_{11}>0\\
(\frac{J_{11}}{d_1}+\frac{J_{22}}{d_2})^2>\frac{4(J_{11}J_{22}-J_{12}J_{21}) }{d_{1}d_{2}} \label{eq:cond2}
\end{eqnarray}
Here $J_{ij}, i,j=1,2$ are as in \eqref{eq:J}
\begin{theorem}
\label{thm:tur}
Consider the two species diffusive model system \eqref{eq:1di}-\eqref{eq:1di2}. There exists a parameter set for which the conditions via \eqref{eq:cond}-\eqref{eq:cond2} hold. Thus there exists Turing instability in model system  \eqref{eq:1di}-\eqref{eq:1di2}.
\end{theorem}

\subsection{Numerical Result:}
Here we demonstrate Turing patterns that form in 1D.  The parameters used in this section are $r               = 0.11, w=1.11, c_{1}=2.81, d = 2.31, w_{1}= 1.32, D=0.1, D_{1}=0.09, d_1=1.0\times 10^{-5}$ and $d_2 =1.0\times 10^{-2}$. The initial condition used is a small perturbation around the positive homogeneous steady state given as 
\begin{align*}
X=X^{*} + \epsilon_1  cos^2(10x)(x > 0)(x < \pi),\\
Y=Y^{*}  + \epsilon_2 cos^2(10x)(x > 0)(x < \pi),\\
\end{align*}
where $ \epsilon_i=0.005$ $\forall i$. 

\begin{figure}[htb]
\subfigure{
\includegraphics[scale =.4]{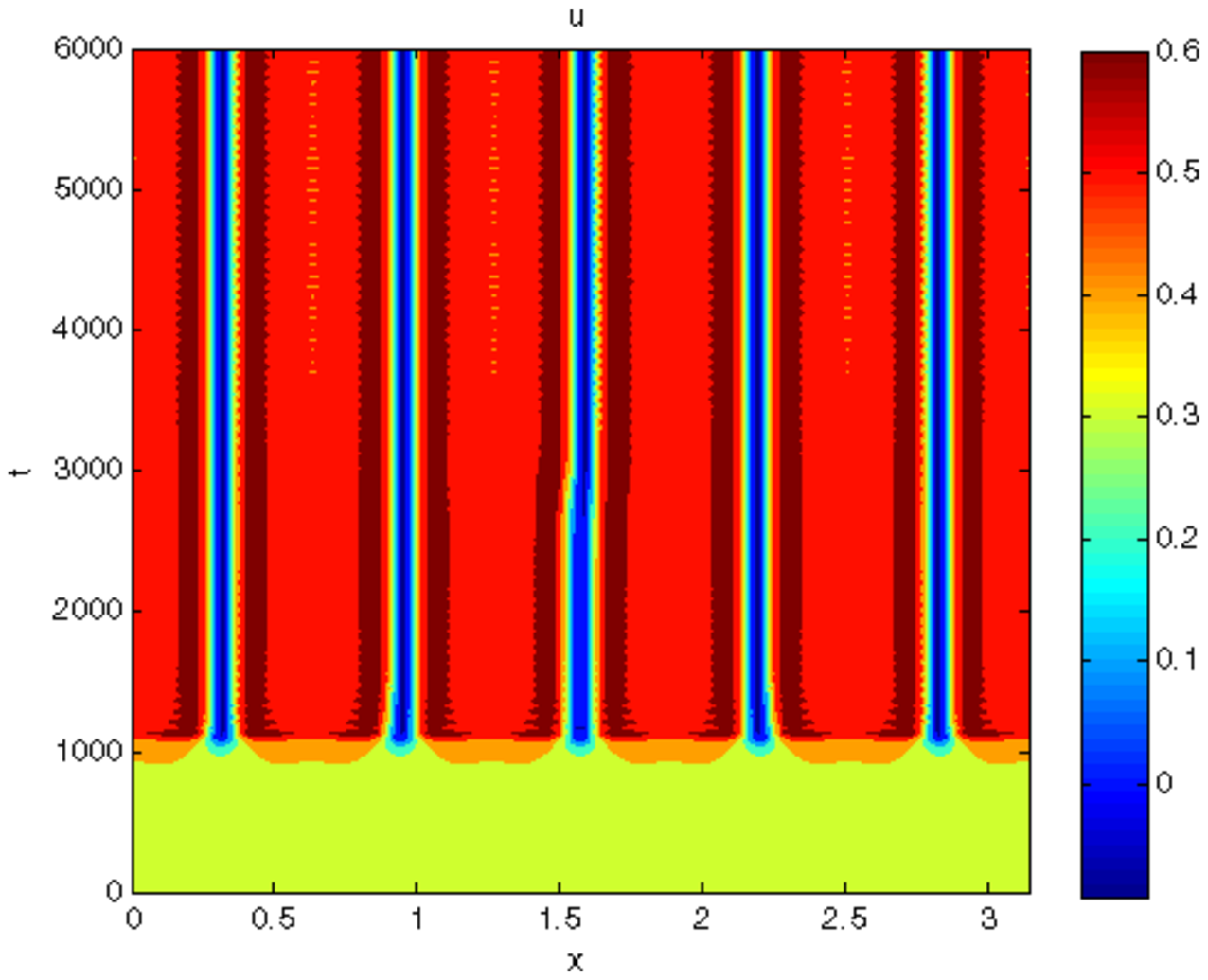}}
\subfigure{
\includegraphics[scale=.4]{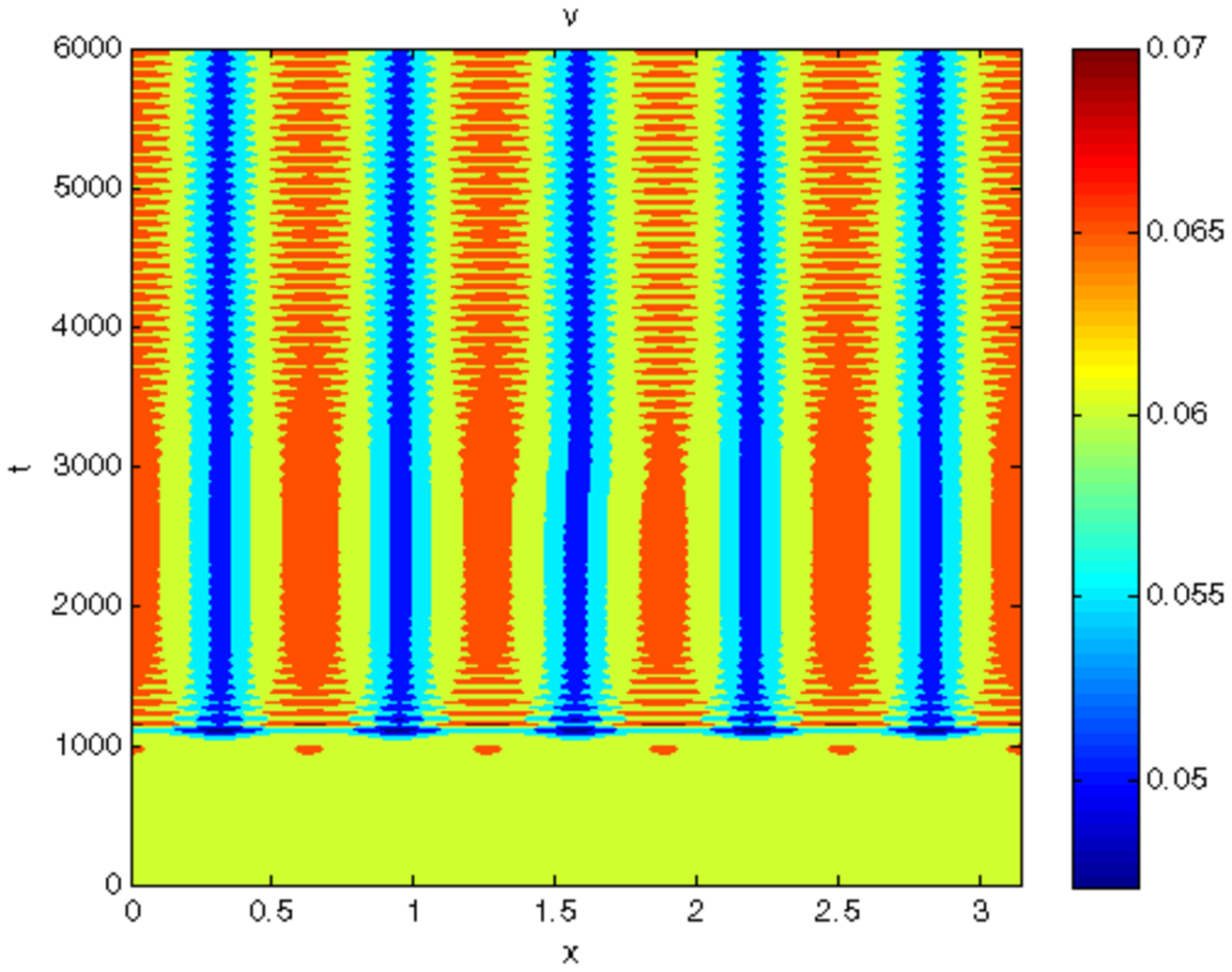}}
\caption{The densities of the species $X, Y$ are shown as contour plots in the x-t plane (1 dimensional in space). The long-time simulation yields stripe patterns.  }
\label{fig:1tt}
\end{figure}
Here we demonstrate Turing patterns that form in 2D. The 2D result in this case were obtained using $R$ package $deSolve $ \cite{Rwork}.The initial condition used is a small perturbation around the positive homogeneous steady state given as 
\begin{align*}
X=X^{*} + \epsilon_1  \cos^2(10x)\cos^2(10y),\\
Y=Y^{*}  + \epsilon_2 \cos^2(10x)\cos^2(10y),\\
\end{align*}
whereas $ \epsilon_i=0.01$ $\forall i$ and $\Omega=[0,L_x]\times [0,L_y]$, with $L_x=L_y=\pi$. 
All numerical simulations are carried out over a $300 \times 300$ lattice with time-step $\Delta t = 0.1$ and spatial steps $\Delta x = \Delta y = 0.01$.
\begin{figure}[!htb]
\begin{center}
\subfigure{
 \includegraphics[scale=.3]{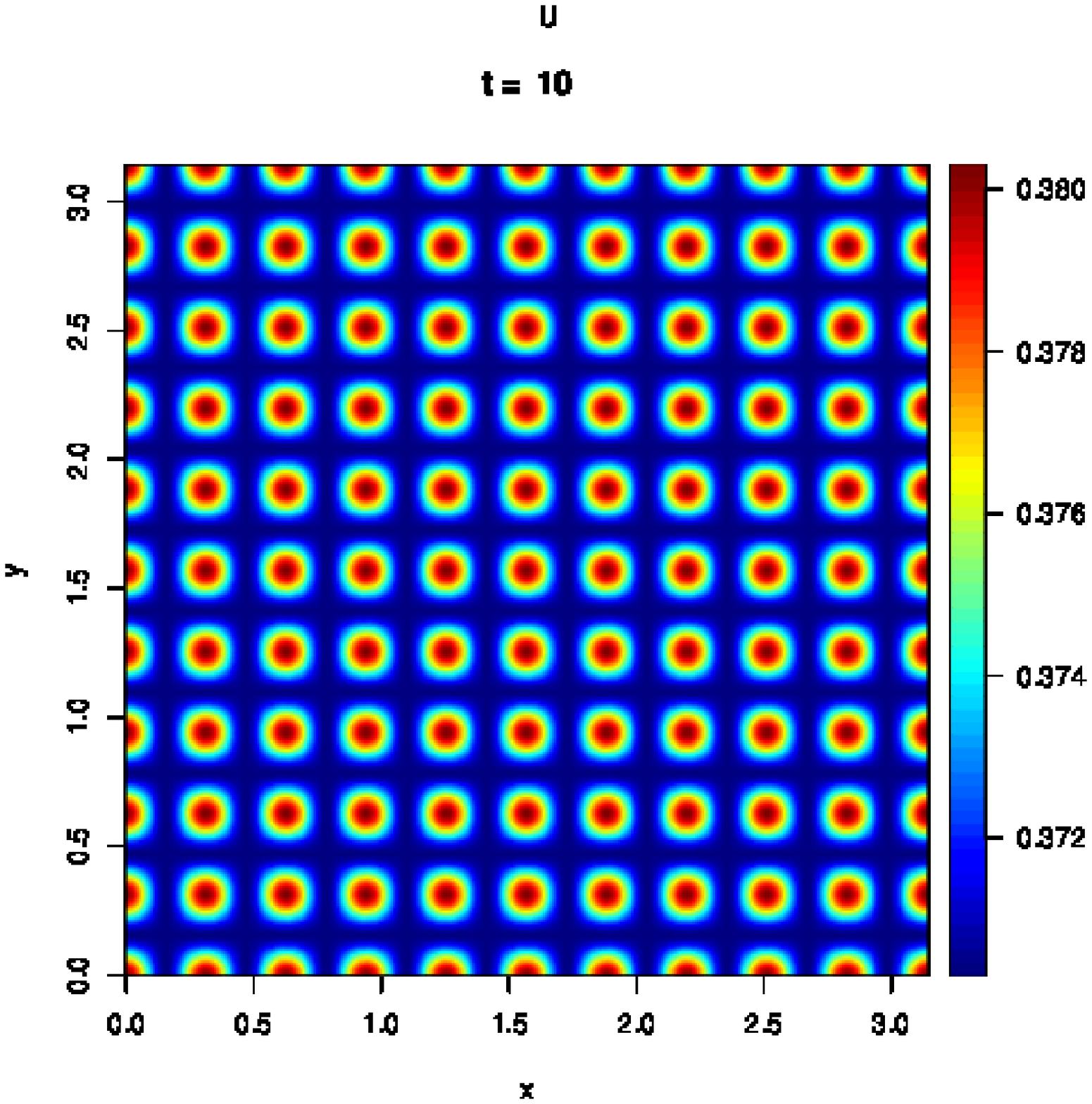}}
 \subfigure{
 \includegraphics[scale=.3]{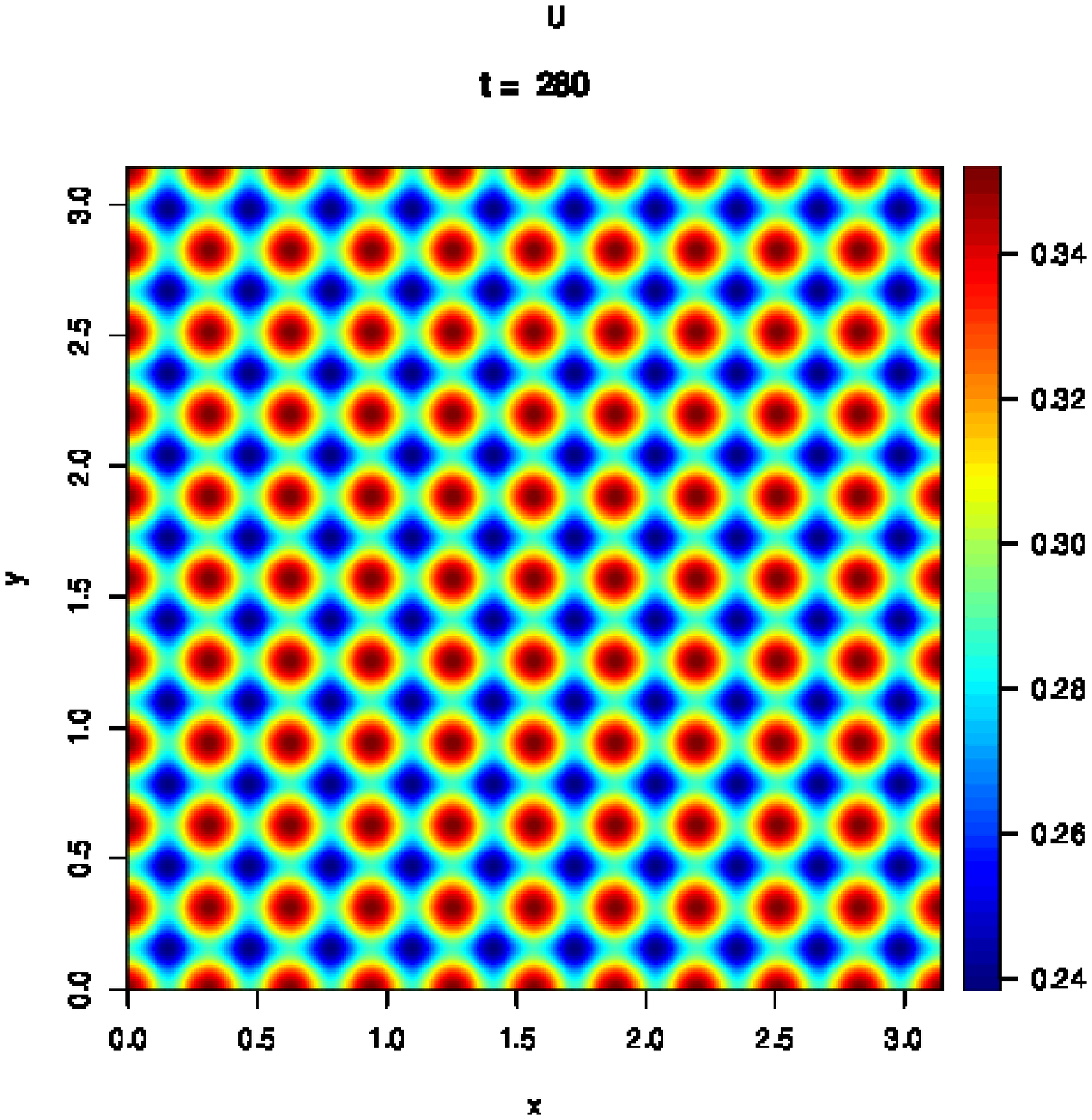}}
  \subfigure{
 \includegraphics[scale=.3]{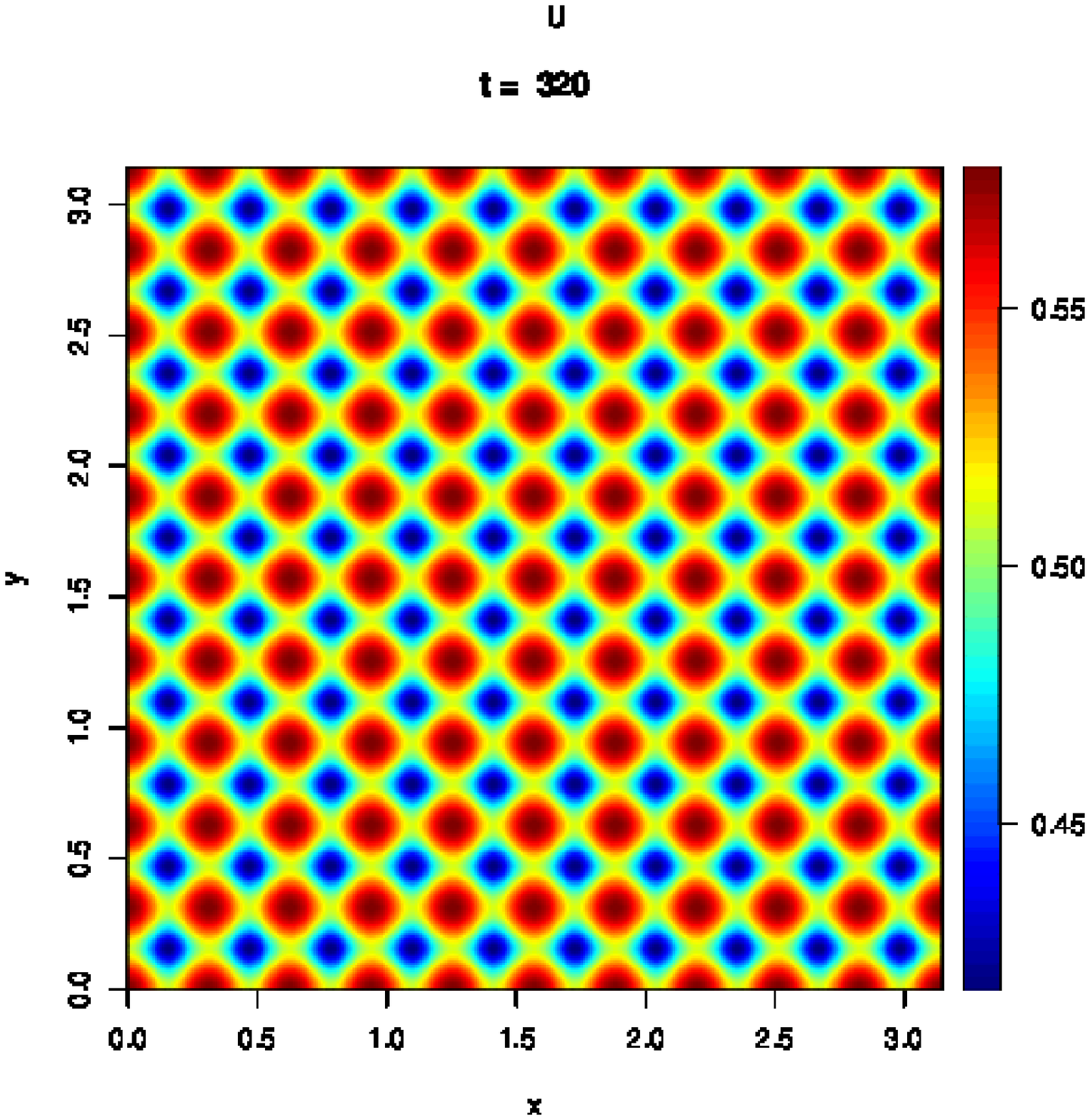}}
   \subfigure{
 \includegraphics[scale=.3]{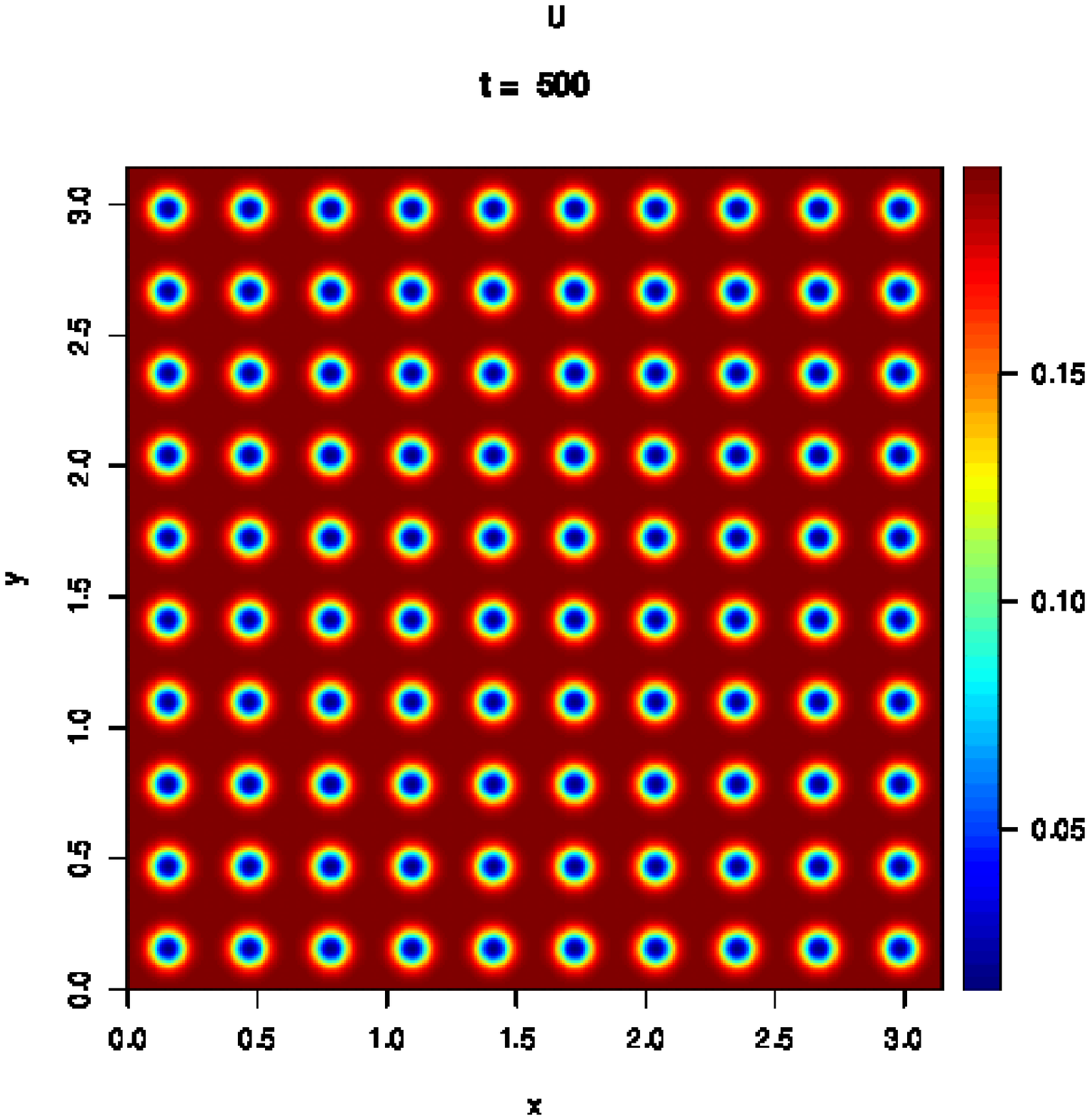}}
\end{center}
 \caption{Snapshots of the time evolution of the prey in 2 dimensions at different times. (a) t=10, (b) t=280, (c) t=320, (d) t=500.}
 \label{fig:2t}
\end{figure}


\begin{figure}[htb]
{
\includegraphics[width =2.5in]{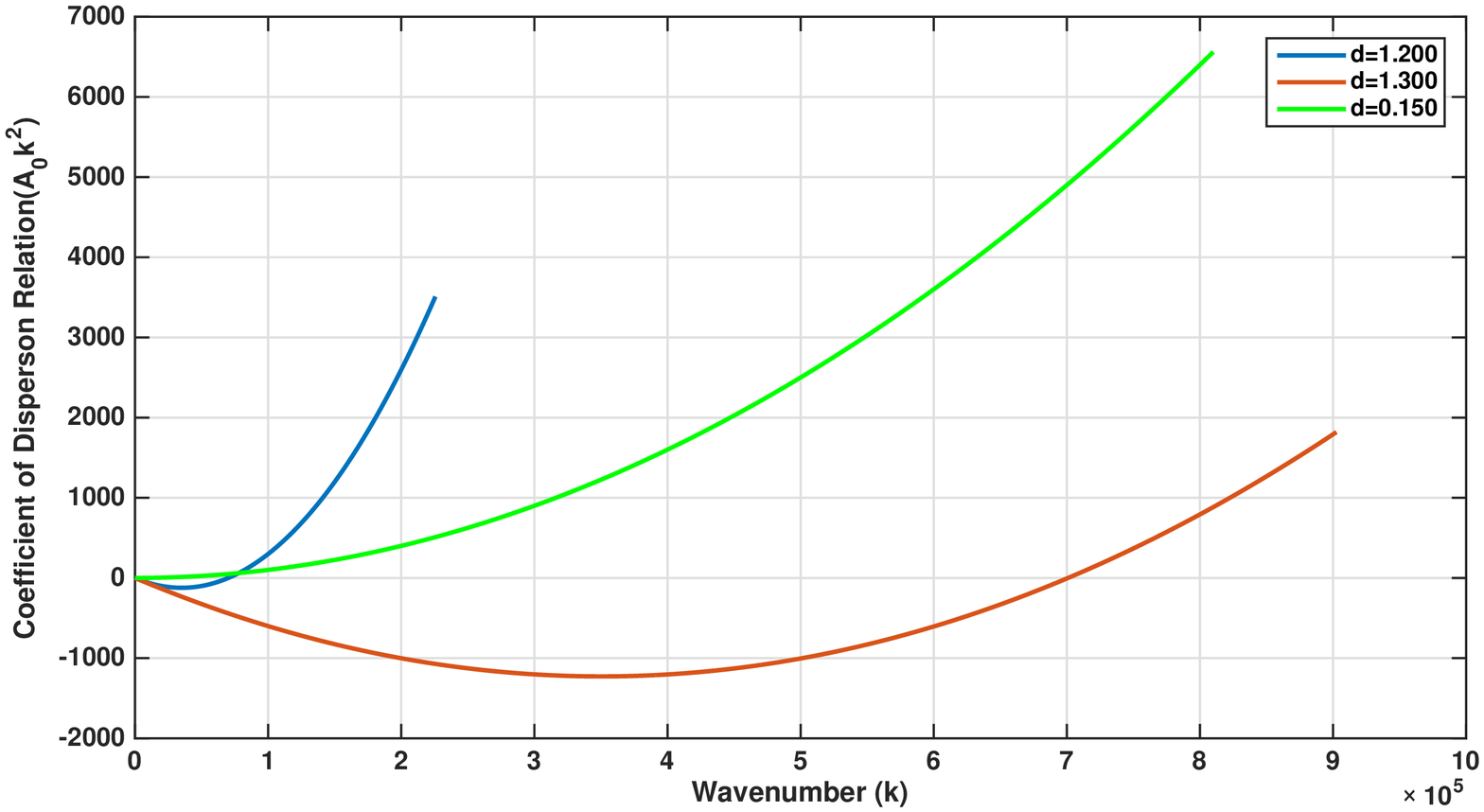}}
\caption{The dispersion plot in the environmental protection parameter $d$ is shown. The plot shows that increasing $d$ enhances the Turing instability.}
\label{1t}
\end{figure}

In Fig.\eqref{1t}, we show the behaviour of $d$ on Turing patterns; whether $d$ can induce Turing patterns. Clearly as seen in Fig.\eqref{1t}, increasing $d$ from $0.15$ to $1.3$ can induce Turing patterns. This is achieved when $A_0(k^2)<0$ for some value of $k^2$. Hence from Fig.\eqref{1t}, $d$ can induce Turing patterns.

\section{Stability Analysis and the Existence of Hopf-bifurcation in delayed diffusive model system}
\label{sec:3}

Upon introducing time delay the dynamics of the diffusive system \eqref{eq:1di}-\eqref{eq:1di2} is governed by the following
systems of equations,
\begin{equation}\label{eq:3.1}
\frac{\partial X}{\partial t}=d_1 \Delta X + X\Bigg[r\left(1-\frac{X}{K}\right)-\frac{\omega  Y}{D+dX+Y}\Bigg],
\end{equation}

\begin{equation}\label{eq:3.1a}
\frac{\partial Y}{\partial t}=d_2 \Delta Y + Y\Bigg[cY-\frac{\omega_1 Y(t-\tau)}{X(t-\tau)+D_1}\Bigg].
\end{equation}

By linearizing system \eqref{eq:3.1}-\eqref{eq:3.1a}, we can rewrite the system \eqref{eq:3.1}-\eqref{eq:3.1a} as follows:

\begin{equation}\label{eq:3.2}
\frac{\partial x}{\partial t}=d \Delta x + L(x(t),x(t-\tau))+N(x(t),x(t-\tau))
\end{equation}
where
$$x(t)=\left(
            \begin{array}{c}
             {\displaystyle{X(t)}}\\
             {\displaystyle{Y(t)}}\\
            \end{array}
     \right),
		d=\left(
            \begin{array}{cc}
             {\displaystyle{d_1}} & {\displaystyle{0}}\\ \\
             {\displaystyle{0}} & {\displaystyle{d_2}}\\
            \end{array}
     \right),$$
		\\
		$$N=\left(
            \begin{array}{c}
             {\displaystyle{-\frac{r}{K}X^2-\frac{\omega XY}{(D+dX+Y)}}}\\
             {\displaystyle{cY^2-\frac{ \omega_1 YY(t-\tau)}{(X(t-\tau)+D_1)}}}\\
            \end{array}
						\right),$$
\\
$$L=\left(
      \begin{array}{cc}
      {\displaystyle{a_{11}}} & {\displaystyle{a_{12}}} \\ \\
      {\displaystyle{0}} & {\displaystyle{cY^*}} \\
    \end{array}
    \right)
					\left(
         \begin{array}{c}
         {\displaystyle{X(t)}}\\
         {\displaystyle{Y(t)}}\\
         \end{array}
		     \right)
		   + \left(
       \begin{array}{cc}
       {\displaystyle{0}} & {\displaystyle{0}}\\ \\
     {\displaystyle{b_{21}}} & {\displaystyle{b_{22}}}\\
    \end{array}
    \right)
				\left(
            \begin{array}{c}
             {\displaystyle{X(t-\tau)}}\\
             {\displaystyle{Y(t-\tau)}}\\
            \end{array}
     \right),$$
		where
		\begin{align*}
		a_{11} &= -\frac{rX^*}{K}+\frac{\omega d Y^*X^*}{(D+dX^*+Y^*)^2},\\
		a_{12} &= -\frac{\omega X^*(D+dX^*)}{(D+dX^*+Y^*)^2},\\
		b_{21} &= \frac{\omega_1 {Y^*}^2}{(X^*+D_1)^2},\\
		b_{22} &= -\frac{\omega_1 Y^*}{(X^*+D_1)}.
		\end{align*}
		So, we can get the characteristic equation of system \eqref{eq:3.2} as follows:
		$$\Bigg|\begin{array}{cc}
		{\displaystyle{\lambda + d_1k^2-a_{11}}} & {\displaystyle{-a_{12}}}\\ \\
		{\displaystyle{-b_{21}e^{-\lambda \tau}}} & {\displaystyle{\lambda + d_2k^2-cY-b_{22}e^{-\lambda \tau}}}\\
		\end{array}\Bigg|=0$$

This is equivalent to the following equation
\begin{eqnarray}\label{eq:3.3}
U_k(\lambda,\tau)=\lambda^2+A_1(k) \lambda +A_0(k) +e^{-\lambda \tau}(B_1(k) \lambda + B_0(k)) = 0
\end{eqnarray}

for $k \in N_0 ={0,1,2,...}$, where,
\begin{align*}
A_1(k)&=d_1 k^2+d_2 k^2-\frac{\omega dY^*X^*}{(D+dX^*+Y^*)^2}+\frac{rX^*}{K},\\
&-cY^*,\\
A_0(k)&=d_1 k^2 d_2 k^2- \frac{d_2 k^2 \omega dY^*X^*}{(D+dX^*+Y^*)^2}+\frac{d_2 k^2 rX^*}{K}\\
&-d_1k^2cY^*+\frac{cY^*\omega dY^*X^*}{(D+dX^*+Y^*)^2}-\frac{cY^*rX^*}{K},\\
B_1(k)&=\frac{\omega_1 Y^*}{X^*+D_1},\\
B_0(k)&=\frac{d_1k^2 \omega_1 Y^*}{X^*+D_1}-\frac{\omega d Y^*X^*\omega_1Y^*}{(D+dX^*+Y^*)^2(X^*+D_1)}\\
&+\frac{\omega_1Y^*rX^*}{K(X^*+D_1)}+\frac{\omega \omega_1X^*{Y^*}^2(D+dX^*)}{(D+dX^*+Y^*)^2(X^*+D_1)^2}.\\
\end{align*}

{\it Case I.} When $\tau=0$, the associated characteristic equation \eqref{eq:3.3} becomes

\begin{align}\label{eq:3.4}
U_k(\lambda,0)=\lambda^2+(A_1(k)+B_1(k))\lambda+(A_0(k)+B_0(k)),
\end{align}
for any $k \in N_0$. By the Routh-Hurwitz criterion, we know that the roots of Eq. \eqref{eq:3.3} have the negative real parts if $A_1(k)+B_1(k) > 0$ and $A_0(k)+B_0(k) > 0$.

{\it Case II.} When $\tau \neq 0$, suppose  $\pm i\omega (\omega>0)$ are the roots of Eq. \eqref{eq:3.3}  for $k \in N_0$. Substituting $\pm i\omega$ into Eq. \eqref{eq:3.3}  and separating real and imaginary parts, we have
$$-\omega^2+A_0(k)=-\omega B_1(k) \mbox{sin}\omega \tau - B_0(k) \mbox{cos} \omega \tau$$
$$\omega A_1(k)=-\omega B_1(k) \mbox{cos}\omega \tau+B_0(k) \mbox{sin}\omega \tau$$
which leads to
\begin{equation}\label{eq:3.5}
\omega^4+({A_1}^2(k)-{B_1}^2(k)-2A_0(k)) \omega^2 + {A_0}^2(k)-{B_0}^2(k)=0
\end{equation}
This Eq.\eqref{eq:3.5} has a unique positive real root\\
$$\omega_0=\Bigg(\frac{-Q + \sqrt{{Q^2-4({A_0}^2(k)-{B_0}^2(k))}}}{2}\Bigg)^{\frac{1}{2}}$$
where $Q=({A_1}^2(k)-{B_1}^2(k)-2A_0(k))$
such that $Q^2>4({A_0}^2(k)-{B_0}^2(k))$ and $Q<0$ and Eq.\eqref{eq:3.3}  has a pair of pure imaginary roots $\pm i\omega$ when
\begin{align*}
\tau=&{\tau_k}^j={\tau_k}^0+\frac{2j\pi}{\omega_0}, \quad j=0,1,2,...,\\
\tau^*&={\tau_k}^0\\
&=\frac{1}{\omega_0}\mbox{arccos}\Bigg(\frac{{B_0}(k)({\omega_0}^2-{A_0}(k))-{\omega_0}^2{B_1}(k){A_1}(k)}{{\omega_0}^2 {B_1}^2(k)+{B_0}^2(k)}\Bigg). \\
\end{align*}
Taking derivative of Eq. \eqref{eq:3.3}  with respect to $\tau$ , we have\\
\begin{align*}
Re\Bigg(\frac{d\lambda(k)}{d\tau}\Bigg)^{-1}\Bigg|_{\tau={\tau_k}^j}=\frac{\omega U\mbox{sin}\omega\tau+\omega^2 V\mbox{cos}\omega\tau-{B_1}^2(k)\omega^2}{{B_1}^2(k)\omega^4+{B_0}^2(k)\omega^2},
\end{align*}
where\\
\begin{align*}
U&=(A_1(k)B_0(k)+2B_1(k)\omega^2),\\
V&=(2B_0(k)-A_1(k)B_1(k)).
\end{align*}
Therefore,\\
$$Re\Bigg(\frac{d\lambda(k)}{d\tau}\Bigg)^{-1}\Bigg|_{\tau={\tau_k}^j}>0$$
if $2B_0(k)>A_1(k)B_1(k)$. Hence, transversality condition holds.

\section{Direction and Stability of spatial Hopf-bifurcation}\label{sec:4}

In this section, we describe the direction, stability and period of the
bifurcating periodic solutions of system \eqref{eq:3.1}-\eqref{eq:3.1a} from the steady state.
The method we use is based on the normal form theory and center
manifold theorem as given by Hassard et al. \cite{BD81}.

Let $x_{1}=X-X^*,~ x_{2}=Y-Y^*$ and $ {x{_{i}}}(t)=x_{i}(\tau t)$ for $i=1, 2$, $\alpha=\tau-\tau^*$ ; the system \eqref{eq:3.1}-\eqref{eq:3.1a} then transforms to partial functional differential equations as
\begin{equation}\label{eq:4.1}
{\dot x}(t)=\tau^* d \Delta x +\tau^* L(x(t))+[\alpha d \Delta x + \alpha L(x(t))+(\alpha+\tau^*)N(x(t))],
\end{equation}

where
$$x(t)=\left(
            \begin{array}{c}
             {\displaystyle{X(t)}}\\
             {\displaystyle{Y(t)}}\\
            \end{array}
     \right),
		d=\left(
            \begin{array}{cc}
             {\displaystyle{d_1}} & {\displaystyle{0}}\\ \\
             {\displaystyle{0}} & {\displaystyle{d_2}}\\
            \end{array}
     \right),$$
		\\
		$$N=\left(
            \begin{array}{c}
             {\displaystyle{-\frac{r}{K}X^2-\frac{\omega XY}{(D+dX+Y)}}}\\
             {\displaystyle{cY^2-\frac{ \omega_1 YY(t-1)}{(X(t-1)+D_1)}}}\\
            \end{array}
						\right),$$

$$L=\left(
      \begin{array}{cc}
      {\displaystyle{a_{11}}} & {\displaystyle{a_{12}}} \\ \\
      {\displaystyle{0}} & {\displaystyle{cY^*}} \\
    \end{array}
    \right)
					\left(
         \begin{array}{c}
         {\displaystyle{X(t)}}\\
         {\displaystyle{Y(t)}}\\
         \end{array}
		     \right)
		   + \left(
       \begin{array}{cc}
       {\displaystyle{0}} & {\displaystyle{0}}\\ \\
     {\displaystyle{b_{21}}} & {\displaystyle{b_{22}}}\\
    \end{array}
    \right)
				\left(
            \begin{array}{c}
             {\displaystyle{X(t-1)}}\\
             {\displaystyle{Y(t-1)}}\\
            \end{array}
     \right).$$
For convenience, we rewrite system \eqref{eq:4.1} in the following form
\begin{equation}\label{eq:4.2}
{\dot x}(t)=\tau^* d \Delta x + L_*(x(t))+N_*(x(t)),
\end{equation}
where
\begin{align*}
L_*(x(t))&=\tau^* L(x(t)),\\
N_*(x(t))&=\alpha d \Delta x + \alpha L(x(t))+(\alpha+\tau^*)N(x(t))
\end{align*}

And associating with the linearizing system of system \eqref{eq:4.2} becomes
\begin{equation}\label{eq:4.3}
{\dot x}(t)=\tau^* d \Delta x + L_*(x(t))
\end{equation}
From the above section, we know that the system \eqref{eq:4.3} has a pair of imaginary eigenvalues $\pm i\omega_0\tau^*$. Consider the ordinary differential equation
\begin{equation}\label{eq:4.4}
{\dot x}(t)=-\tau^*k^2dx(t)+L_*(x(t)).
\end{equation}
By the Reisz Representation theorem, there exists a function $\eta(\theta,\tau^*)$ whose components are of bounded variation for $\theta~\in~[-1,0]$
such that
\begin{equation}\label{eq:4.5}
L_{*}(\phi)-\tau^*k^2d\phi (0)~=~
\int_{-1}^{0}~\phi(\theta)d\eta(\theta),\quad\mbox{for} ~~\phi ~\in~ C.
\end{equation}
In fact, we can choose\\
\begin{align}\label{eq:4.6}
\eta(\theta,\tau^*) =\left\{\begin{array}{lll}\tau^*\left(
      \begin{array}{cc}
      {\displaystyle{a_{11}-d_1k^2}} & {\displaystyle{a_{12}}} \\ \\
      {\displaystyle{0}} & {\displaystyle{cY^*-d_2k^2}} \\
    \end{array}
    \right),& \theta=0,\\
\\ 0,&  \theta \in (-1,0), \\
\\ -\tau^*\left(
       \begin{array}{cc}
       {\displaystyle{0}} & {\displaystyle{0}}\\ \\
     {\displaystyle{b_{21}}} & {\displaystyle{b_{22}}}\\
    \end{array}
    \right),& \theta=-1.\end{array}\right.
    \end{align}

For $\phi \in C^{1}([-1,0],R^{2})$ and $\psi \in C^{1}([0,1],(R^2)^*)$, define $A$ and $A^*$ as
$$ A\phi(\theta) =\left\{\begin{array}{ll}\displaystyle{\frac{d \phi(\theta)}{d \theta} },& \theta \in [-1,0),\\
\\ \int_{-1}^{0}~\phi(\theta)d\eta(\theta),& \theta  = 0, \end{array}\right.$$
$$A^{*} \psi (s) = 
\left\{\begin{array}{ll}-\displaystyle{\frac{d \psi (s)}{d s}},& s \in (0,1], \\
 \\ \int_{-1}^{0}~ \psi (-\xi)d {\eta} (\xi),& s=0.
  \end{array}\right. $$
Then $A$ and $A^*$ are a pair of adjoint operators under the following bilinear inner product
\begin{equation}\label{eq:4.7}
\langle \psi ,\phi  \rangle = \bar{\psi} (0) \phi (0)-
\int_{\theta = -1}^{0} \int_{\xi = 0}^{\theta} \bar{\psi} (\xi -\theta) d \eta (\theta) \phi (\xi) d \xi .
\end{equation}

Since $\pm i\omega_{0}
\tau^{*}$ are the eigenvalues of $A$ and $A^*$ respectively.
We need to compute eigenvectors of $A$ and $A^*$ crresponding to
 $+i\omega_{0}\tau^{*}$ and $-i\omega_{0}\tau^{*}$ respectively.\\
 \noindent Suppose $q(\theta)=(\alpha_1,\alpha_2)^{T}e^{i\omega_{0}\tau^{*}\theta}$
 be the eigenvector of $A$ corresponding to eigenvalue $i\omega_{0}\tau^{*}$ then\\
 \begin{equation}\label{eq:4.8}
 Aq(\theta)=i\omega_{0}\tau^{*}q(\theta),
 \end{equation}
 which, for $\theta=0$, gives
\begin{align}\label{eq:4.9}
\tau^{*}\left(
 \begin{array}{cc}
  {i\omega_0-a_{11}-d_1k^2} & {-a_{12}}\\ \\
     {-b_{21} e^{-i\omega_0 \tau^*}} & {i\omega_0-cY^*-d_2k^2-b_{22} e^{-i\omega_0 \tau^*}}\\
  \end{array}
  \right)q(0)
	=\left(
    \begin{array}{cc}
    0 \\
    0\\
\end{array}
\right).
\end{align}

By choosing $\alpha_1=1$, we can write the above equation as follows
\begin{align}\label{eq:4.10}
\left(
\begin{array}{cc}
   {i\omega_0 -a_{11}-d_1k^2} & {-a_{12}}\\ \\
   {\-b_{21}e^{-i\omega_0 \tau^*}} & {i\omega_0-cY^*-d_2k^2-b_{22}e^{-i\omega_0 \tau^*}}\\
  \end{array}
  \right)
  \left(
  \begin{array}{c}
  {\displaystyle{1}}\\
  {\displaystyle{\alpha_2}}\\
  \end{array}
  \right)=\left(
  \begin{array}{c}
  {\displaystyle{0}}\\
  {\displaystyle{0}}\\
  \end{array}
  \right).
  \end{align}
Thus, we obtain
$$\alpha_2=\frac{i\omega_0-a_{11}-d_1k^2}{a_{12}}.$$

Similarly, let $q^*(s)=N(\alpha_{1}^*,\alpha_{2}^*)e^{i\omega_{0}\tau^{*}s}$ be the eigenvectors
of $A^*$ corresponding to eigenvalue $-i\omega_{0}\tau^{*}$, we have
\begin{equation}\label{eq:4.11}
A^*q^*(s)=-i\omega_{0}\tau^{*}q^*(s),
\end{equation}
where $\alpha_{1}^*=1,\alpha_{2}^*=-\frac{i\omega_0+a_{11}+d_2k^2}{b_{21} e^{-i\omega_0 \tau^*}} $.

Using the normalization condition, i.e., $\langle q^*(s),q(\theta)\rangle=1$, we need to choose a suitable value of $N$.
From Eq. \eqref{eq:4.11}, we have

\begin{align*}
\langle q^*(s) ,q(\theta)  \rangle
&= \bar N (1,\bar {{\alpha_2}^*})(1,\alpha_2)^T
-\int_{\theta = -1}^{0} \int_{\xi = 0}^{\theta} \bar N (1,\bar {{\alpha_2}^*}) e^{-i\omega_0 \tau^* (\xi -
\theta)} d \eta (\theta) (1,\alpha_2)^T  e^{i\omega_0 \tau^* \xi} d \xi ,\\
&=\bar N \{ 1 + \alpha_2 \bar{\alpha_2^*} + \tau^*  (b_1 +\alpha_2b_2)\bar{\alpha_2^*}e^{-i\omega_0 \tau^*} \},
\end{align*}
which gives
$$\bar N=\frac{1}{1 + \alpha_2 \bar{\alpha_2^*} + \tau^*(b_1+\alpha_2b_2)\bar{\alpha_2^*}e^{-i\omega_0 \tau^*}}.$$
Thus we can choose $N$ as
$$\frac{1}{1 + \bar{\alpha_2} {\alpha_2^*} + \tau^*(b_1+\bar{\alpha_2}b_2){\alpha_2^*}e^{i\omega_0 \tau^*}}$$
such that
$$\langle q^*(s) ,q(\theta)  \rangle =1, \langle q^*(s) ,\bar {q}(\theta)  \rangle=0.$$
In this case, $\langle \Psi,\Phi \rangle=I$, where $\Phi=(q(\theta),\bar{q}(\theta)), \Psi=(q^*(s),{\bar{q}}^*(s))^T$ and $I$ is the unit matrix.\\
In the following, we mainly use the theory from Wu \cite{WU96}. Let $P=\mbox{span}\{q(\theta),\bar{q}(\theta)\}$ and $P^*=\mbox{span}\{q^*(s),{\bar{q}}^*(s)\}$, then $P$ is the center subspace of Eq. \eqref{eq:4.3} and $P^*$ is the adjoint subspace.\\
In addition, $f_k=(\beta_k^1,\beta_k^2)$ and $c.f_k=c_1\beta_k^1+c_2\beta_k^2$, for $c=(c_1,c_2)^T\in C^2$. The center subspace of system \eqref{eq:4.3} is given by $P_{CN}\Im$, where
\begin{equation}\label{eq:4.12}
P_{CN}\Im (\phi)=\Phi(\Psi,\langle \phi,f_k\rangle).f_k,~~~\phi \in \Im,
\end{equation}
and $\Im=P_{CN}\Im\oplus P_S \Im$, where $P_S \Im$ denotes the complement subspace of $P_{CN} \Im$ in $\Im$. Let $A_{\tau}^*$ be the infinitesimal generator induced by the solution of system \eqref{eq:4.3}. Then the system \eqref{eq:4.1} can be rewritten in the abstract form as
$$\dot{x}(t)=A_{\tau^*}x(t)+R(\alpha,x(t))$$
where
$$R(\alpha,x(t))=\left\{\begin{array}{ll}0,& \theta \in [-1,0), \\
 \\ N(\alpha,x(t)),& \theta=0. \end{array}\right. $$

Using the decomposition $\Im=P_{CN}\Im\oplus P_S \Im$ and Eq.\eqref{eq:4.12}, the solution of system \eqref{eq:4.2} can be written as\\
$$x(t)=\Phi \left(
    \begin{array}{c}
  {\displaystyle{x_1(t)}} \\
  {\displaystyle{x_2(t)}}\\
\end{array}
\right).f_k+h(x_1,x_2,\alpha),$$
where $(x_1(t),x_2(t))^T=(\psi,\langle x(t),f_k \rangle)$, and $h(x_1,x_2,\alpha) \in P_S\Im,h(0,0,0)=0, dh(0,0,0)=0$. In particuiar, the solution of system \eqref{eq:4.2} on the center manifold is given by\\
\begin{align}\label{eq:4.13}
x(t)=\Phi \left(
    \begin{array}{c}
  {\displaystyle{x_1(t)}} \\
  {\displaystyle{x_2(t)}}\\
\end{array}
\right).f_k+h(x_1,x_2,0).
\end{align}


Let $z=x_1-ix_2, \psi(0)=(\Psi_1(0),\Psi_2(0))^T$, and notice that $p_1=\Phi_1+i\Phi_2$, then\\
$$\Phi \left(
    \begin{array}{c}
  {\displaystyle{x_1(t)}} \\
  {\displaystyle{x_2(t)}}\\
\end{array}
\right).f_k=(\Phi_1,\Phi_2) \left(
    \begin{array}{c}
  {\displaystyle{\frac{z+\bar{z}}{2}}} \\
  {\displaystyle{\frac{i(z+\bar{z})}{2}}}\\
\end{array}
\right).f_k\\
=\frac{1}{2}(p_1z+{\bar{p}}_1 \bar z).f_k.$$\\
So system \eqref{eq:4.13} can be transformed into
\begin{equation}\label{eq:4.14}
x(t)=\frac{1}{2}(p_1z+{\bar{p}}_1 \bar z).f_k+w(z, \bar z),
\end{equation}
where
\begin{align*}
w(z,\bar z) &= h(\frac{z+\bar z}{2},\frac{i(z-\bar z)}{2},0)\\
&= w_{20}(\theta)\frac{z^2}{2}+w_{11}(\theta)z \bar z + w_{02}(\theta)\frac{{\bar z}^2}{2}\\
&+w_{21}(\theta)\frac{z^2 \bar z}{2}+\ldots~.
\end{align*}
Moreover, $z$ satisfies
\begin{equation}\label{eq:4.15}
\dot z=i\omega_0 \tau^* z+g(z,\bar z),
\end{equation}
where
\begin{align}\label{eq:4.16}
g(z,\bar z)&=(\Psi_1(0),\Psi_2(0)) \langle N_*(x(t),0),f_k \rangle \nonumber \\
&=g_{20}\frac{z^2}{2}+g_{11}z \bar z + g_{02}\frac{{\bar z}^2}{2}+g_{21}\frac{z^2 \bar z}{2}+\ldots~.
\end{align}

From system \eqref{eq:4.2} and \eqref{eq:4.13}, it follows that
\begin{align*}
x_{1t}(0)&=z + \bar{z}
+W_{20}^{1} (0) \frac{z^2}{2} + W_{11}^{(1)} (0) z \bar{z}\\
&+W_{02}^{(1)}(0) \frac{{\bar{z}}^2}{2} + \ldots ~,\\
x_{2t}(0) &=\alpha_2 z + \bar{\alpha_2} \bar{z} +W_{20}^{(2)}
(0) \frac{z^2}{2} + W_{11}^{(2)} (0) z \bar{z}\\
&+W_{02}^{(2)} (0)\frac{{\bar{z}}^2}{2} +  \ldots ~, \\
x_{1t}(-1) &= z e^{- i\omega_0
\tau^{*}}+  \bar{z}e^{i\omega_0
\tau^{*}} +W_{20}^{(1)} (-1) \frac{z^2}{2} \\
&+ W_{11}^{(1)} (-1) z
\bar{z} + W_{02}^{(1)} (-1)\frac{{\bar{z}}^2}{2} + \ldots ~, \\
 x_{2t}(-1) &=\alpha_2 z e^{- i\omega_0
\tau^{*}}+ \bar{\alpha_2} \bar{z}e^{i\omega_0
\tau^{*}} +W_{20}^{(2)} (-1) \frac{z^2}{2}\\
&+ W_{11}^{(2)} (-1) z\bar{z} + W_{02}^{(2)} (-1)\frac{{\bar{z}}^2}{2} +\ldots~. 
\end{align*}

Now, from Eq.\eqref{eq:4.16}, we have
\begin{align}\label{eq:4.17}
g(z,\bar z)&=\tau^{*}\bar N(1,\bar{\alpha_2^*})
\left(
                \begin{array}{c}
              {-\frac{r}{K}\phi_{1}^2(0)-\frac{\omega\phi_{1}(0)\phi_{2}(0)}{(D+d\phi_{1}(0)+\phi_{2}(0))}}\nonumber\\
              {c\phi_{2}^2(0)-\frac{ \omega_1 \phi_{2}(0)\phi_{2}(-1)}{(\phi_{1}(-1)+D_{1})}}
              \end{array}
              \right),\nonumber\\
&=\tau^* \bar N\Bigg[\left\lbrace\left(-\frac{r}{K}-\frac{\omega \alpha_2}{D}\right)+\bar{\alpha_2^*}\alpha_2^2\left(c-\frac{ \omega_1 e^{-i\omega_0 \tau^*}}{D_1}\right)\right\rbrace z^2\nonumber\\
&+\Bigg\{ \left(-\frac{r}{K}-\frac{\omega Re(\alpha_2)}{D}\right)+\bar{\alpha_2^*} \alpha_2 \bar{\alpha_2}\left(c-\frac{ \omega_1 cos(\omega_0 \tau^*)}{D_1}\right)\Bigg\} z \bar{z}\nonumber\\
&+\left\lbrace\left(-\frac{r}{K}-\frac{\omega \bar{\alpha_2}}{D}\right)+\bar{\alpha_2^*} \bar{\alpha_2}^2 \left(c-\frac{ \omega_1 e^{i\omega_0 \tau^*}}{D_1}\right)\right\rbrace {\bar z}^2\\
&+\Bigg\{ -\frac{r}{K}(\alpha_2  W_{11}^{(1)} (0)+ \frac{W_{20}^{(1)} (0)}{2})+\bar{\alpha_2^*}c(\frac{\bar{\alpha_2}W_{20}^{(2)} (0)}{2}\nonumber\\
&+\alpha_2 W_{11}^{(2)} (0))+\frac{\bar{\alpha_2^*} \omega_1}{D_1}(\alpha_2 W_{11}^{(2)} (-1)+\frac{\bar{\alpha_2} W_{20}^{(2)} (0)}{2}e^{i \omega_0 \tau^*}\nonumber\\
&+\alpha_2 W_{11}^{(2)} (0) e^{-i \omega_0 \tau^*}+\frac{\bar{\alpha_2}}{2}W_{20}^{(2)} (-1))+\frac{\omega}{D}(W_{11}^{(2)} (0)\nonumber\\
&+\frac{\bar{\alpha_2}W_{20}^{(1)} (0)}{2}+\frac{W_{20}^{(2)} (0)}{2}+\alpha_2W_{11}^{(1)} (0))\Bigg\}{z^2 \bar z}\Bigg].\nonumber
\end{align}

Comparing the coefficient with \eqref{eq:4.20}, we have
\begin{equation}\label{eq:4.18}
g_{20}=2\tau^* \bar N \left(\left(-\frac{r}{K}-\frac{\omega \alpha_2}{D}\right)+\bar{\alpha_2^*}\alpha_2^2\left(c-\frac{ \omega_1 e^{-i\omega_0 \tau^*}}{D_1}\right)\right),
\end{equation}
\begin{equation}\label{eq:4.19}
g_{11}=\tau^* \bar N \Bigg(\left(-\frac{r}{K}-\frac{\omega Re(\alpha_2)}{D}\right)+\bar{\alpha_2^*} \alpha_2 \bar{\alpha_2}\Bigg(c-\frac{ \omega_1 cos(\omega_0 \tau^*)}{D_1}\Bigg)\Bigg),
\end{equation}
\begin{equation}\label{eq:4.20}
~g_{02}=2\tau^* \bar N \left(\left(-\frac{r}{K}-\frac{\omega \bar{\alpha_2}}{D}\right)+\bar{\alpha_2^*} \bar{\alpha_2}^2 \left(c-\frac{ \omega_1 e^{i\omega_0 \tau^*}}{D_1}\right)\right),
\end{equation}
\begin{align}\label{eq:4.21}
g_{21}&=2\tau^* \bar N\Bigg\{-\frac{r}{K}(\alpha_2  W_{11}^{(1)} (0)+ \frac{W_{20}^{(1)} (0)}{2})+\bar{\alpha_2^*}c(\frac{\bar{\alpha_2}W_{20}^{(2)} (0)}{2}\\
&+\alpha_2 W_{11}^{(2)} (0))+\frac{\bar{\alpha_2^*} \omega_1}{D_1}(\alpha_2 W_{11}^{(2)} (-1)+\frac{\bar{\alpha_2} W_{20}^{(2)} (0)}{2}e^{i \omega_0 \tau^*}\nonumber\\
&+\alpha_2 W_{11}^{(2)} (0) e^{-i \omega_0 \tau^*}+\frac{\bar{\alpha_2}W_{20}^{(2)} (-1)}{2})+ \frac{\omega}{D}(W_{11}^{(2)} (0)\nonumber\\
&+\frac{\bar{\alpha_2}W_{20}^{(1)} (0)}{2}+\frac{W_{20}^{(2)} (0)}{2}+\alpha_2W_{11}^{(1)} (0))\Bigg\}.\nonumber
\end{align}
In order to compute $g_{21}$, we need to compute them $W_{20}^{(i)}(\theta)$ and $W_{11}^{(i)}(\theta), i=1,2$.\\
From Eq.\eqref{eq:4.14}, we have
\begin{align*}
\dot{W} &= \dot{x_{t}} - \dot{z}q - \dot{\bar{z}}\bar{q}\\
\end{align*}
\begin{align}\label{eq:4.22}
\dot{W}&= \left\{\begin{array}{ll}
A W - 2 {\mbox Re} \{{\bar{q}}^* (0) f_{0} q(\theta)\}, & \theta \in [-1,0), \\
A W - 2 {\mbox Re} \{{\bar{q}}^* (0) f_{0} q(\theta)\} + f_{0}, & \theta = 0,
\end{array}\right.
\end{align}
\begin{align}\label{eq:4.23}
\dot{W} &= A W + H (z, \bar{z}, \theta),
\end{align}
with\\
\begin{equation}\label{eq:4.24}
H(z,\bar z,\theta)=H_{20}(\theta)\frac{z^2}{2}+H_{11}(\theta) z \bar z+H_{02}(\theta)\frac{z^2}{2}+\ldots~~.
\end{equation}
Also, on $C_0$, using chain rule, we get
\begin{equation}\label{eq:4.25}
 \dot{W}=W_z \dot{z}+W_{\bar z} \dot{\bar z}.
\end{equation}
It follows from Eqs.\eqref{eq:4.14}, \eqref{eq:4.23} and \eqref{eq:4.25}
\begin{equation}\label{eq:4.26}
(A-2i\omega_0 \tau^*)W_{20}=-H_{20}(\theta),
\end{equation}
\begin{equation}\label{eq:4.27}
 AW_{11}=-H_{11}(\theta).
\end{equation}
\\
 Now for $\theta \in [-1,0)$, we have
 \begin{align}\label{eq:4.28}
 H(z,\bar z,\theta)&=-\bar{q^*}(0) f_0 q(\theta)-q^*(0) \bar{f_0} \bar{q}(\theta)\nonumber\\
 &=-g(z,\bar z)q(\theta)-\bar g(z,\bar z)\bar q(\theta),\\
 &=-(g_{20}q(\theta)+{\bar{g}}_{02}\bar q(\theta))\frac{z^2}{2}-(g_{11}q(\theta)\nonumber\\
 &+{\bar g}_{11} {\bar q}(\theta))z \bar z+\ldots ~,\nonumber
 \end{align}
 
 which on comparing the coefficients with \eqref{eq:4.24} gives\\
 \begin{equation}\label{eq:4.29}
  H_{20}(\theta)=-g_{20} q(\theta)-{\bar g}_{02} \bar q (\theta),
  \end{equation}
  \begin{equation}\label{eq:4.30}
  H_{11}(\theta)=-g_{11} q(\theta)-{\bar g}_{11} \bar q (\theta).
  \end{equation}
 \\
From Eqs.\eqref{eq:4.26}, \eqref{eq:4.29} and the definition of $A$, we have\\
\begin{equation}\label{eq:4.31}
{\dot W}_{20}(\theta)=2i\omega_0 \tau^* W_{20}(\theta)+g_{20}q(\theta)+{\bar g}_{02} {\bar q}(\theta).
\end{equation}
Note that $q(\theta)=q(0)e^{i\omega_0 \tau^* \theta}$, we have\\
\begin{equation}\label{eq:4.32}
W_{20}(\theta)=\frac{ig_{20}}{\omega_0 \tau^*} q(0)e^{i\omega_0 \tau^* \theta}.f_k+\frac{i {\bar g}_{20}}{3\omega_0 \tau^*} \bar q (0)e^{-i\omega_0 \tau^* \theta}.f_k
+E_1 e^{2i \omega_0 \tau^* \theta}.
\end{equation}
Similarly, from Eq. \eqref{eq:4.27}, \eqref{eq:4.30} and the definition of $A$, we have\\
\begin{align}\label{eq:4.33}
{\dot W}_{11}(\theta)&=g_{11}q(\theta)+{\bar g}_{11} {\bar q}(\theta),
\end{align}
\begin{align}\label{eq:4.34}
W_{11}(\theta)&=-\frac{ig_{11}}{\omega_0 \tau^*} q(0)e^{i\omega_0 \tau^* \theta}.f_k+\frac{i {\bar g}_{11}}{\omega_0 \tau^*} \bar q (0)e^{-i\omega_0 \tau^* \theta}.f_k+E_2,
\end{align}
where $E_1=\big(E^{(1)}_1,E^{(2)}_1\big)$ and $E_2=\big(E^{(1)}_2,E^{(2)}_2\big)\in R^2$ are constant vectors, to be determined.\\
It follows from the definition of $A$ and connecting Eqs. \eqref{eq:4.14},  \eqref{eq:4.32} and  \eqref{eq:4.34}, we have\\
\\
\begin{align}\label{eq:4.35}
&2i\omega_0 \tau^*\Bigg(\frac{ig_{20}}{\omega_0 \tau^*} q(0)e^{i\omega_0 \tau^* \theta}.f_k+\frac{i {\bar g}_{20}}{3\omega_0 \tau^*} \bar q (0)e^{-i\omega_0 \tau^* \theta}.f_k+E_1 e^{2i \omega_0 \tau^* \theta}\Bigg)\nonumber\\
&-\Delta\Bigg(\frac{ig_{20}}{\omega_0 \tau^*} q(0)e^{i\omega_0 \tau^* \theta}.f_k+\frac{i {\bar g}_{20}}{3\omega_0 \tau^*} \bar q (0)e^{-i\omega_0 \tau^* \theta}.f_k+E_1 e^{2i \omega_0 \tau^* \theta}\Bigg)\nonumber\\
&-L_*\Bigg(\frac{ig_{20}}{\omega_0 \tau^*} q(0)e^{i\omega_0 \tau^* \theta}.f_k+\frac{i {\bar g}_{20}}{3\omega_0 \tau^*} \bar q (0)e^{-i\omega_0 \tau^* \theta}.f_k+E_1 e^{2i \omega_0 \tau^* \theta}\Bigg)\\
&=- g_{20} q(0) - \bar{g}_{02} \bar{q} (0)\nonumber\\
&+ 2 \tau^{*}
    \left(
    \begin{array}{c}
  {\displaystyle{-\frac{r}{K}-\frac{\omega \alpha_2}{D}}} \\
  {\displaystyle{\alpha_2^2\left(c-\frac{ \omega_1 e^{-i\omega_0 \tau^*}}{D_1}\right)}}\\
\end{array}
\right),\nonumber
\end{align}

and
\begin{align}\label{eq:4.36}
&-\Delta\Bigg(-\frac{ig_{11}}{\omega_0 \tau^*} q(0)e^{i\omega_0 \tau^* \theta}.f_k+\frac{i {\bar g}_{11}}{\omega_0 \tau^*} \bar q (0)e^{-i\omega_0 \tau^* \theta}.f_k+E_2\Bigg)\nonumber\\
&-L_*\Bigg(-\frac{ig_{11}}{\omega_0 \tau^*} q(0)e^{i\omega_0 \tau^* \theta}.f_k+\frac{i {\bar g}_{11}}{\omega_0 \tau^*} \bar q (0)e^{-i\omega_0 \tau^* \theta}.f_k+E_2\Bigg)\\
& = - g_{11} q(0) - \bar{g}_{11} \bar{q} (0)\nonumber\\
&+  \tau^{*}
   \left(
   \begin{array}{c}
  {\displaystyle{-\frac{r}{K}-\frac{\omega Re(\alpha_2)}{D}}} \\
  {\displaystyle{\alpha_2 \bar{\alpha_2}\left(c-\frac{ \omega_1 cos(\omega_0 \tau^*)}{D_1}\right)}}\\
\end{array}
\right).\nonumber
\end{align}

Meanwhile noting the following equalities\\
\begin{align}\label{eq:4.37}
\tau^*d\Delta(q(0).f_k)+L(\tau^*)(q(0).f_k)=i\omega_0 \tau^*(q(0).f_k),
\end{align}
\begin{align}\label{eq:4.38}
\tau^*d\Delta(\bar q(0).f_k)+L(\tau^*)(\bar q(0).f_k)=-i\omega_0 \tau^*(\bar q(0).f_k),
\end{align}
and connecting Eqs.\eqref{eq:4.35} to \eqref{eq:4.38}, we can obtain the following equalities\\

\begin{align}\label{eq:4.39}
&2i\omega_0\tau^*E_1-\tau^*d\Delta E_1-L_*(E_1e^{2i\omega_0 \tau^* \theta})\nonumber\\
&=2\tau^{*}
    \left(
    \begin{array}{c}
  {\displaystyle{-\frac{r}{K}-\frac{\omega \alpha_2}{D}}} \\
  {\displaystyle{\alpha_2^2\left(c-\frac{ \omega_1 e^{-i\omega_0 \tau^*}}{D_1}\right)}}\\
\end{array}
\right),\nonumber\\
&-\tau^*d\Delta E_1-L_*(E_2)
\end{align}

\begin{align}\label{eq:4.40}
=\tau^{*}
   \left(
   \begin{array}{c}
  {\displaystyle{-\frac{r}{K}-\frac{\omega Re(\alpha_2)}{D}}} \\
  {\displaystyle{\alpha_2 \bar{\alpha_2}\left(c-\frac{ \omega_1 cos(\omega_0 \tau^*)}{D_1}\right)}}\\
\end{array}
\right)
\end{align}
Thus from \eqref{eq:4.39} and \eqref{eq:4.40}, we can get the values of $E_1$ and $E_2$\\

\begin{align}\label{eq:4.41}
&\left(
\begin{array}{cc}
  {\displaystyle{2i\omega_0-a_{11}-d_1k^2}} & {\displaystyle{-a_{12}}}\\ \\
  {\displaystyle{-b_{21} e^{-i\omega_0 \tau^*}}} & {\displaystyle{2i\omega_0-cY^*-d_2k^2-b_{22}e^{-i\omega_0 \tau^*}}}\\
  \end{array}
  \right)E_1\nonumber\\
  &= 2\left(
     \begin{array}{c}
  {\displaystyle{-\frac{r}{K}-\frac{\omega \alpha_2}{D}}} \\
  {\displaystyle{\alpha_2^2\left(c-\frac{ \omega_1 e^{-i\omega_0 \tau^*}}{D_1}\right)}}\\
\end{array}
\right),
\end{align}
\\
\begin{align}\label{eq:4.42}
&\left(
   \begin{array}{cc}
 {\displaystyle{-a_{11}-d_1k^2}} & {\displaystyle{-a_{12}}}\\
     {\displaystyle{-b_{21} e^{-i\omega_0 \tau^*}}} & {\displaystyle{-cY^*-d_2k^2-b_2 e^{-i\omega_0 \tau^*}}}\\
 \end{array}
 \right)E_{2} \nonumber\\
& = \left(
     \begin{array}{c}
   {\displaystyle{-\frac{r}{K}-\frac{\omega Re(\alpha_2)}{D}}} \\
  {\displaystyle{\alpha_2 \bar{\alpha_2}\left(c-\frac{ \omega_1 cos(\omega_0 \tau^*)}{D_1}\right)}}\\
\end{array}
\right).
\end{align}
\\
Thus, we can determine $W_{20} (\theta)$ and $W_{11} (\theta)$. Furthermore, we get the value of $g_{21}$. According to the theorem developed by Wu \cite{WU96}, we can compute the following values:

\begin{align}\label{eq:43}
&c_{1} (0) = \frac{i}{2 \omega_0 \tau^{*}} (g_{20} g_{11} - 2 |g_{11}|^2 - \frac{|g_{02}|^2}{3}) + \frac{g_{21}}{2}, \nonumber\\
&\mu_{2} = - \frac{{\mbox Re} \{c_{1} (0)\}}{{\mbox Re} \{\lambda^{'} (\tau^{*})\}}, \nonumber\\
&\beta_{2} = 2 {\mbox Re} \{c_{1} (0)\}, \\
&T_{2} = - \frac{{\mbox Im} \{c_{1} (0)\} + \mu_{2} {\mbox Im}\{\lambda^{'} (\tau^{*})\}}{\omega_0\tau^{*}}, \nonumber
\end{align}

which determine the properties of bifurcating periodic solution in
the center manifold at the critical value $\tau^{*}$.

\begin{theorem}
\label{thm:hop}
For the expressions given in \eqref{eq:43}, the following results hold:
\begin{itemize}
\item The sign of $\mu_{2}$~ determines the direction of the
Hopf bifurcation. If $\mu_{2} > 0$, then the Hopf bifurcation is
supercritical and the bifurcating periodic solutions exist for
$\tau > \tau^{*}$. If $\mu_{2} < 0$, then the Hopf bifurcation
is subcritical and the bifurcating periodic solutions exist for
$\tau < \tau^{*}$,
\item The parameter $\beta_{2}$~ determines the stability of the
bifurcating periodic solutions. The bifurcating periodic solutions
are stable if $\beta_{2} < 0$ and unstable if $\beta_{2} > 0$,
\item  Also, $T_{2}$~ determines the period of the bifurcating periodic
solutions. The period of the bifurcating periodic
solutions increases if $T_{2} > 0$ and decreases if
$T_{2} < 0$.
\end{itemize}
\end{theorem}

By applying these results, we can get the properties of the Hopf-bifurcation. For the parametr set considered in Section 4.1, we have calculated these parameters and observed that $\mu_2<0$ which imply that the Hopf bifurcation will be subcritical and the bifurcating periodic solutions exist for $\tau < \tau^{*}$. Also we get $\beta_2>0$ and $T_2>0$, i.e., bifurcating periodic solutions are unstable with increasing period.

\section{Discussion and Conclusion}
In the current note we show that the model system \eqref{eq:1d} proposed in \cite{RK15}, as well as its ``0 time lag" counterpart (the $\tau=0$, or no delay case) both blow-up in finite time, for sufficiently large initial data. Thus, there is no global stability for any parameter range  for this model system. For there to exist a global in time solution, one must restrict initial data to being close to equilibrium. It would be an interesting question to investigate how close this needs to be. There is probably a delicate interplay here between the parameters in the problem, and the size of this invariant set. Recently, Joglekar et al. (\cite{JYO14}, \cite{JYO15}) discuss this issue by defining $ " \epsilon$-uncertaintity" in case of dynamical system and explained that how a small perturbation in parameter value changes the asymptotic behavior of the system. This is still an interesting open question. Also, for the delayed problem, one does not need $cY^2$, for blow-up to occur. If we choose $cY^m$, $1 <m<2$, blow-up can \emph{still} occur. This is shown in Figure \ref{fig:1t1}. This makes the delayed problem, completely different from the non delayed one, because in system  \eqref{eq:1} if the predator is modeled as $cY^m$, $1 <m<2$, finite time blow-up \emph{cannot} occur. Future investigations might include stochastic effects, or the coupled effects of diffusion and time delay.\\



\end{document}